\documentclass [12pt]{article}

\usepackage{latexsym,amssymb,amsfonts,amsmath,amsthm,graphicx}
\usepackage{graphicx}
\usepackage{hyperref}

\newtheorem{lemma}{Lemma}[section]
\newtheorem{proposition}{Proposition}[section]
\newtheorem{corollary}{Corollary}[section]
\newtheorem{theorem}{Theorem}[section]
\newtheorem{remark}{Remark}[section]

\numberwithin{equation}{section}

\newcommand{\s}{\vspace{2ex}}
\newcommand{\n}{\noindent}
\newcommand{\eps}{\varepsilon}
\newcommand{\R}{\mathbb{R}}
\newcommand{\Z}{\mathbb{Z}}

\newcommand{\E}{\mathbb{E}}
\newcommand{\PP}{\mathbb{P}}

\title{Almost-Sure Stability of the Single Mode Solution of a Noisy Autoparametric Vibration Absorber}

\author{{Peter H. Baxendale} \and
{N. Sri Namachchivaya}}

\AtEndDocument{\bigskip{\footnotesize
  \noindent
  \textsc{Peter H. Baxendale. Department of Mathematics, University of Southern California, CA 90089, USA.}
  \textit{E-mail address}: \texttt{\href{mailto:baxendal@usc.edu}{baxendal@usc.edu}}
  \par
  \addvspace{\medskipamount}
  \noindent
  \textsc{N. Sri Namachchivaya. Department of Applied Mathematics, University of Waterloo, ON N2L3G1, Canada.}
  \textit{E-mail address}: \texttt{\href{mailto:navam@uwaterloo.ca}{navam@uwaterloo.ca}}
}}

\begin{document}

\maketitle

\begin{abstract}

For a pendulum suspended below a vibrating block with white noise forcing, the solution in which the pendulum remains vertical is called the single mode solution.  When this solution becomes unstable there is energy transfer from the block to the pendulum, helping to absorb the vibrations of the block.   We study the Lyapunov exponent $\lambda$ governing the almost-sure stability of the process linearized along the single mode solution.  The linearized equation is excited by a combination of white and colored noise processes, which makes the evaluation of $\lambda$ non trivial.  For white noise forcing of intensity $\eps \nu$ we prove $\lambda(\eps) = \lambda_0+ \eps^2 \lambda_2 + O(\eps^4)$ as $\eps \to 0$, where $\lambda_0$ and $\lambda_2$ are given explicitly in terms of the parameters of the system.

\end{abstract}

\s

\n Keywords: autoparametric system, Lyapunov exponent, stability boundary, resonance.

\section{Introduction} 

Autoparametric systems are vibrating systems that consist of two subsystems.  The primary system is an oscillator that is directly excited by some external forcing, and the secondary system is coupled nonlinearly to the oscillator in such a way that it can be at rest while the oscillator is vibrating. 
 
  \begin{figure}[h] \label{fig block pend}
\begin{center}
\includegraphics[scale=1.2]{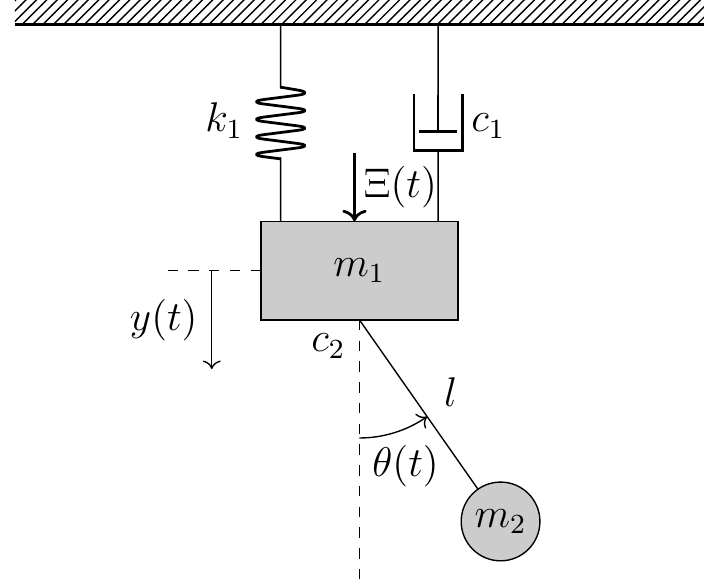}
\end{center}
 \caption{Block and pendulum}
 \end{figure}
In this paper we consider the two-degree-of-freedom system shown in Figure \ref{fig block pend}.  The primary system consists of a block of mass $m_1$ attached by a linear spring of stiffness $k_1$ and a viscous dashpot with a damping coefficient $c_1$.   The block is excited directly by (downward) forcing $\Xi(t)$.  The secondary system comprises a simple pendulum of length $\ell$ attached to the block.  The mass of the pendulum bob is $m_2$ and the pendulum motion is damped by a viscous damper of coefficient $c_2$.  This system appears in Hatwal, Mallik and Ghosh \cite{HMG83}.  See also \cite{BCJ94, BBD96} and a similar model with a non-linear spring and a compound pendulum in \cite{WarKec06}. A closely related system with a vertically mounted cantilever beam with a tip mass appears in Haxton and Barr \cite{HaxBar72}.

It is a characteristic property of an autoparametric system that there is a solution involving forced motion of the primary system while the secondary system remains at rest.  Here it corresponds to the ``single mode'' solution in which the pendulum remains directly below the vibrating block at all times.   As parameters in the system are varied, it may happen that the single mode solution becomes unstable and the pendulum begins to move.  In this case the motion of the primary system causes excitation of the secondary
system, and there is energy transfer from the vibrating block to the pendulum. In effect, 
the pendulum acts as a vibration absorber.
Our primary concern is the almost-sure stability analysis of the single mode solution when the forcing function $\Xi(t)$ is a multiple of white noise.

\s
 
In addition to the block and pendulum and block and beam examples mentioned above, autoparametric systems are used to model the dynamics of structural and mechanical systems such as the vibration of an initially deflected shallow arch \cite{TSB94}, in-plane and out-of-plane motions of suspended cables \cite{BRV}, and the pitching and rolling motions of a ship \cite{NMM73, NR09}.  A larger collection of examples is given in the monograph \cite{TRVN}.
 
Periodically excited autoparametric systems have been studied extensively; see for example \cite{HaxBar72, NayMoo79, HMG83, CarRob88, BCJ94, BBD96}.   These are deterministic systems, and the most interesting situations occur when the natural frequencies of the primary~(excited mode) and the secondary~(unexcited) systems are in 2:1 resonance.   Issues of bifurcation theory arise, and numerical computations frequently show chaotic behavior.  

Autoparametric systems with stochastic forcing have also been studied. 
Gaussian and non-Gaussian closure techniques were used in \cite{IbrRob77, Rob80, IbrHeo86, ChaIbr97, LeeCho00} to approximate the moments of solutions of the associated Fokker-Planck equation.  Sufficient conditions for almost-sure stability of the single mode solution were obtained in \cite{Aria91}.  Stochastic averaging techniques were used in \cite{SKA07}.  

\s

The equations of motion for the block and pendulum system shown in Figure \ref{fig block pend} are 
\begin{equation}
\begin{split}
(m_1+m_2) \ddot y(t) + c_1 \dot y(t) + k_1 y(t) - m_1 \ell (\ddot \theta(t) \sin\theta (t) + {\dot \theta}^2(t) \cos \theta(t) ) &=\Xi(t) \\
m_2 \ell^2 \ddot \theta(t) + c_2 \dot \theta(t) +  m_2 \ell ( g  - \ddot y(t) ) \sin \theta(t) &=0,
\label{eq motion}
\end{split}
\end{equation}
where $y$ represents the downwards displacement of the block relative to its rest (unforced) position, and $\theta$ is the angle of the pendulum, see \cite{HMG83}.  We specialize to the case of white noise forcing, that is $\Xi(t) = \hat{\nu} \dot{W}(t)$ where $W$ is a standard Wiener process and $\hat{\nu}$ represents the noise intensity. 

 Letting $\eta = y/\ell$ represent the dimensionless position of the block, the equations of motion \eqref{eq motion} can be rewritten as
\begin{equation}
\label{sys}
  \begin{split} 
\ddot{\eta}(t) + 2 \zeta_1 \dot {\eta}(t) + \chi^2 \eta(t) - R \big(\ddot{\theta}(t) \sin{\theta}(t)  + {\dot{\theta}}^2(t) \cos{\theta}(t)\big) &= \nu \dot{W}(t)\\
 \ddot{\theta}(t) +  2 \zeta_2  \dot {\theta}(t) +  \big(\kappa^2  - \ddot{\eta}(t)  \big)\sin{\theta}(t) &= 0,
\end{split}
\end{equation}
where  
the constants $\zeta_1$ and $\zeta_2$ are scaled damping coefficients, $\chi^2 = k_1/(m_1+m_2)$, and $\kappa= \sqrt{g/\ell}$ is the frequency of the undamped pendulum.  The parameter $R= m_2/(m_1+m_2)$, $ 0 < R < 1$, represents the mass ratio, and $\nu = \hat{\nu}/[\ell(m_1+m_2)]$ is the scaled noise intensity.

The rigorous interpretation of \eqref{sys} as a 4-dimensional stochastic differential equation (SDE) is carried out in Section \ref{sec posed}.  See Theorem \ref{thm well-posed} which guarantees existence and uniqueness and exponential moments for solutions with arbitrary initial conditions.  
     The single mode solution of the system \eqref{sys} is of the form
\begin{equation}\label{sin mode}
\big(\eta(t), \dot \eta(t)\big)= \big(\bar\eta(t), \dot {\bar\eta}(t)\big
 )  \quad \text{and}\quad \big(\theta(t), \dot \theta(t) \big)=\big(0, 0\big),
\end{equation}
 where the motion  $\{(\bar{\eta}(t),\dot{\bar{\eta}}(t)):t \ge 0\} $ of the block is described by
\begin{equation}\label{bareta}
 \ddot{\bar{\eta}}(t)+2 \zeta_1 \dot{\bar{\eta}}(t)  + \chi^2 \bar{\eta}(t)  = \nu \dot{W}(t).
\end{equation}
The single mode \eqref{bareta} has a stationary solution $\{(\bar{\eta}(t),\dot{\bar{\eta}}(t)):t \ge 0\} $ with Gaussian stationary probability measure with mean $(0,0)$ and covariance matrix $R = \dfrac{\nu^2}{4 \zeta_1 }\begin{pmatrix} 1/\chi^2 & 0 \\ 0 & 1 \end{pmatrix}$.   

The stability or instability of the single mode solution is governed by the large time behavior of the process  obtained by linearizing the 4-dimensional SDE along the single mode solution $(\bar{\eta}(t),\dot{\bar{\eta}}(t), 0, 0)$.  In particular it is determined by the sign of the Lyapunov exponent
  \begin{equation}
\label{lam intro}
\lambda :=\lim \limits_{t\to\infty}\frac{1}{t}\log\|\left(\varphi(t), \dot\varphi(t) \right)\|
\end{equation} 
associated with the linear SDE
     \begin{equation}
    \ddot{\varphi}(t)   + 2 \zeta_2 \dot{\varphi}(t)   + \big(\kappa^2 +2 \zeta_1 \dot{\bar{\eta}}(t)   + \chi^2 \bar{\eta}(t)\big) \varphi(t)  = \nu \varphi(t)   \dot{W}(t) \label{phi bareta}  
 \end{equation}
where $(\bar{\eta}(t),\dot{\bar{\eta}}(t)$ is given by \eqref{bareta}, see Section \ref{sec stab sing}.  

The novelty of the linear SDE \eqref{phi bareta} is that it is parametrically excited by the colored noise processes $\bar{\eta}(t)$ and $\dot{\bar{\eta}}(t)$ as well as the white noise (generalized) process $\dot{W}(t)$ which is driving them in \eqref{bareta}.   The techniques and results developed in this paper apply to a larger class of excitations of \eqref{phi bareta} consisting of combinations of Gaussian colored noise and white noise processes.  This larger class is described in Section \ref{sec more gen}. 

 It is shown in Section \ref{sec Khas} that the right side of \eqref{lam intro} exists as an almost-sure limit, and is given by Khas'minskii's integral formula \eqref{lam Khas} for all initial values $(\bar{\eta},\dot{\bar{\eta}}, \varphi,\dot{\varphi})$ with $(\varphi,\dot{\varphi}) \neq (0,0)$.  
 Although the integrand $Q(v,\psi)$ in \eqref{lam Khas} is given explicitly, the measure $m$ is characterized only as the unique invariant probability measure for an associated diffusion process, and it is hard to evaluate the integral and obtain an exact formula for the Lyapunov exponent.  However, assuming the pendulum is underdamped, so that the damped frequency $\kappa_d = \sqrt{\kappa^2 - \zeta_2^2}$ is real and positive, a simple bound on $Q$ gives the following upper bound on the Lyapunov exponent, 
  \begin{equation} \label{lam upper} 
 \lambda \le - \zeta_2 + \frac{\nu}{2 \kappa_d}\sqrt{\frac{\chi^2+4 \zeta_1^2}{4 \zeta_1}} + \frac{\nu^2}{2 \kappa_d^2}.
\end{equation}
see Proposition \ref{prop lam upper}.  This gives almost-sure stability for sufficiently small noise intensity $\nu$.

The main results in the paper are given in Section \ref{sec scaling}.  Theorem \ref{thm lam} and Propositions \ref{prop Vdot} and \ref{prop psd xi} give rigorous asymptotic estimates for the Lyapunov exponent $\lambda(\eps)$ for the system obtained when the noise intensity is rescaled $\nu \rightsquigarrow \eps \nu$. For the block and pendulum system we show
       \begin{equation} \label{lam est}
      \lambda(\eps)  = -\zeta_2 + \eps^2 \frac{2 \kappa_d^2 \nu^2}{(\chi^2-4 \kappa_d^2)^2+16 \zeta_1^2 \kappa_d^2} + O(\eps^4)
      \end{equation}
as $\eps \to 0$.   Section \ref{sec scaling2} contains results in the case when the noise intensity and the pendulum damping are both rescaled: $\nu \rightsquigarrow \eps \nu$ and $\zeta_2 \rightsquigarrow \eps^2 \zeta_2$.  There is a brief discussion of the similarities and differences in the 2:1 resonance effects for white noise and  periodic forcing.   The final three sections contain the proofs.

\section{Well-posedness for the nonlinear system} \label{sec posed}

Since $0 < R < 1$, the equations \eqref{sys} can be rearranged to give 
     \begin{equation}
    \begin{split}
    \ddot{\eta}(t)&+ \frac{2 \zeta_1 \dot{\eta}(t) + \chi^2 \eta(t) -R \dot{\theta}^2(t) \cos \theta(t) + 2R\zeta_2 \dot{\theta}(t) \sin \theta(t) +\kappa^2R \sin^2 \theta(t)}{(1-R\sin^2 \theta(t))} \\
    & \hspace{40ex}=  \frac{\nu}{(1-R\sin^2 \theta(t))}\dot{W}(t)   
      \\
   \ddot{\theta}(t) &+ \frac{2 \zeta_2 \dot{\theta}(t) + \kappa^2 \sin \theta(t) +2 \zeta_1 \dot{\eta}(t) \sin \theta(t) +\chi^2 \eta \sin \theta(t) -R \dot{\theta}^2(t) \sin \theta(t) \cos \theta(t) }{(1-R\sin^2 \theta(t))}\\ 
   &  \hspace{40ex}=  \frac{ \nu \sin \theta(t)}{(1-R\sin^2 \theta(t))} \dot{W}(t). \end{split} 
   \label{sys2}
\end{equation} 
  This second order system for $(\eta,\theta)$ can now be turned into a first order system of stochastic differential equations for the process $(\eta,\dot{\eta},\theta,\dot{\theta})$ taking values in the space $N = \R^2 \times \R/(2 \pi \Z) \times \R$.  Write
   $$(v_1,v_2,u_1,u_2) = (\eta,\dot{\eta},\theta, \dot{\theta}).
   $$
Then \eqref{sys2} gives
   \begin{equation}
     \begin{split}
    dv_1 & = v_2\,dt\\ 
    dv_2 & = \frac{-2 \zeta_1 v_2 - \chi^2 v_1 +R u_2^2 \cos u_1 - 2R\zeta_2 u_2 \sin u_1 -R \kappa^2 \sin^2 u_1}{(1-R\sin^2 u_1)} dt \\
    & \hspace{10ex} + \frac{\nu}{(1-R\sin^2 u_1)}dW(t) \\
     du_1 & = u_2 \,dt\\
  du_2 & =   \frac{-2 \zeta_2 u_2 - \kappa^2 \sin u_1 -2 \zeta_1 v_2 \sin u_1 - \chi^2 v_1 \sin u_1 + R u_2^2\sin u_1 \cos u_1 }{(1-R\sin^2 u_1)}dt \\
  & \hspace{10ex} +  \frac{ \nu \sin u_1}{(1-R\sin^2 u_1)} dW(t).
   \end{split} \label{sde}
\end{equation} 
We see that $du_1(t)dW(t) = 0$ and so the It\^o and Stratonovich interpretations of \eqref{sde} are identical.

 Let ${\cal L}$ denote the generator of the diffusion $\{(v_1(t),v_2(t),u_1(t),u_2(t)): t \ge 0\}$ given by \eqref{sde}.  For ease of notation write $U(t) = (v_1(t),v_2(t),u_1(t),u_2(t)) \in N$.  Noting that $|u_1(t)| \le 1$ we abuse notation by writing $\|U(t)\| = \sqrt{v_1^2(t)+v_2^2(t)+u_2^2(t)}$. 
  
\begin{theorem}  \label{thm well-posed} (i) For every initial condition $U(0) \in N$  the system \eqref{sde} has a unique solution $\{U(t): t \ge 0\}$ valid for all $t \ge 0$.  

(ii) There exists $b > 0$ such that for every initial $U(0)$ the process $\{U(t): t \ge 0\}$ satisfies the condition $\|U(t)\| \le b$ for infinitely many $t$ as $t \to \infty$. 

(iii) There exist $\beta > 0$ and $K$ such that every invariant probability $\pi$ for \eqref{sde} satisfies 
      $$
      \int_N \exp\bigl(\beta \|U\|^2\bigr) d\pi(U) \le K.
      $$ 

(iv) There exist $0 < \beta \le \beta_1 $ and $K_1$ and for every $\gamma > 0$ there exists $K_2$ such that
     $$
     \E\exp \bigl(\beta \|U(t)\|^2\bigr) \le K_1 e^{-\gamma t} \exp\left(\beta_1 \|U(0)\|^2\right) + K_2
     $$
for all initial conditions $U(0) \in N$ and all $t \ge 0$. 
\end{theorem}

The proof of Theorem \ref{thm well-posed} is given in Section \ref{sec wellposed}.

\begin{remark} \label{beam}   The corresponding calculation for white noise forcing of the beam model discussed in {\rm \cite{IbrRob77, Rob80, LeeCho00}} fails because it involves division by a term of the form $1-R \theta^2$ rather than $1-R \sin^2 \theta$.  There is no guarantee that the full non-linear system, for example \cite[eqns (4,5)]{IbrRob77}, has a solution valid for all time $t \ge 0$. 
\end{remark}

\begin{remark} \label{rem compound}  Suppose that the simple pendulum is replaced by a compound pendulum with mass $m_2$, moment of inertia $I$ about the pivot point, and distance $d$ of the center of mass from the pivot point.  The equations of motion are now 
   \begin{equation}
\begin{split}
(m_1+m_2) \ddot y(t) + c_1 \dot y(t) + k_1 y(t) - m_2 d \bigl(\ddot \theta(t) \sin\theta (t) + {\dot \theta}^2(t) \cos \theta(t) \bigr) &=\Xi(t) \\
I \ddot \theta(t) + c_2 \dot \theta(t) +  m_2 d ( g  - \ddot y(t) ) \sin \theta(t) &=0.
\label{eq motion comp}
\end{split}
\end{equation} 
Letting $\eta = y/L$ where $L = I/(m_2d)$ is the effective length of the compound pendulum the equations of motion \eqref{eq motion comp} can be written in the form \eqref{sys} where now $\kappa = \sqrt{g/L}$ is the frequency of the compound pendulum, and $R:=(d/L) \cdot m_2/(m_1+m_2) < 1$ because $L \ge d$.
\end{remark}

 There exists at least one stationary solution, the single mode solution described in (\ref{sin mode},\ref{bareta}) above.  In addition there is an ``unstable single mode'' solution in which the pendulum is at rest directly above the block.  It is described by \eqref{bareta} together with $(\theta(t),\dot{\theta}(t)) = (\pi,0)$.  There may or may not be other stationary solutions.  We conjecture the existence of other ``nonlinear'' solutions depends on the stability of the single mode solution.
  
\subsection{Stability of the single mode solution} \label{sec stab sing}

In order to determine the stability of the single mode solution we study the long time behavior of the linearization of the system along the single mode solution (\ref{sin mode},\ref{bareta}).  Writing $\eta(t) = \bar{\eta}(t) + \delta x(t) $ and $\theta(t) = 0 + \delta \varphi(t)$ in \eqref{sys} and letting $\delta \to 0$ gives the set of variational equations:
 \begin{equation}
 \begin{split}  \ddot{x}(t) + 2 \zeta_1 \dot{x}(t)  + \chi^2 x(t)  &  =  0   \\
  \ddot{\varphi}(t)   + 2 \zeta_2 \dot{\varphi}(t)   + \big(\kappa^2 -\ddot{\bar{\eta}}(t)\big) \varphi(t)   & =  0,
 \end{split}\label{lin}
\end{equation}
where $\bar{\eta}_t$ is given by \eqref{bareta}.
More rigorously, 
linearizing the system \eqref{sde} along the solution $(\bar{\eta},\dot{\bar{\eta}},0,0)$ gives a first order system for $(x,\dot{x},\varphi,\dot{\varphi})$ which is equivalent to 
 \begin{equation}
 \begin{split}  \ddot{x}(t) + 2 \zeta_1 \dot{x}(t)  + \chi^2 x(t)  &  =  0   \\
  \ddot{\varphi}(t)   + 2 \zeta_2 \dot{\varphi}(t)   + \big(\kappa^2 +2 \zeta_1 \dot{\bar{\eta}}(t)  + \chi^2 \bar{\eta}(t)\big) \varphi(t)   & =  \nu \varphi(t)   \dot{W}(t).
 \end{split}\label{lin2}
\end{equation}
We see that the second line in \eqref{lin} together with the information about $\bar{\eta}(t)$ in \eqref{bareta} is rigorously interpreted as the second line in \eqref{lin2}.

The stability of the single mode solution is determined by the long time behavior of the linearized process $\{(x(t),\dot{x}(t),\varphi(t),\dot{\varphi}(t)): t \ge 0\}$.  The top equation in \eqref{lin2} is damped and unforced, so   $(x(t),\dot{x}(t)) \to (0,0)$ as $t\to\infty$.  Therefore it is enough to consider the parametric exitation of $\{(\varphi(t),\dot{\varphi}(t)): t \ge 0\}$ caused by the single mode vibration $\bar{\eta}$.  For simplicity of notation we replace $\bar{\eta}$ with $\eta$ and consider the long term growth or decay rate for the process $\{(\varphi(t),\dot{\varphi}(t)): t \ge 0\}$ given by
      \begin{align}
     \ddot{\eta}(t) +2 \zeta_1 \dot{\eta}(t)  + \chi^2 \eta(t) &  = \nu \dot{W}(t) \label{eta}\\
     \ddot{\varphi}(t)   + 2 \zeta_2 \dot{\varphi}(t)   + \big(\kappa^2 +2 \zeta_1 \dot{\eta}(t)   + \chi^2 \eta(t)\big) \varphi(t)  &   = \nu \varphi(t)   \dot{W}(t), \label{phi eta}  
 \end{align}
The stability of the single mode solution is determined by the almost-sure exponential growth rate of~\eqref{phi eta}, expressed precisely as the Lyapunov exponent 
\begin{equation}
\label{E:lambda2}
\lambda :=\lim \limits_{t\to\infty}\frac{1}{t}\log\|\left(\varphi(t), \dot\varphi(t) \right)\|.
\end{equation}
We will see in Corollary \ref{cor ergodic} that the almost-sure limit in \eqref{E:lambda2} exists and takes the same value for all initial conditions $(\eta(0),\dot{\eta}(0),\varphi(0),\dot{\varphi}(0))$ with $(\varphi(0),\dot{\varphi}(0)) \neq (0,0)$.

\begin{remark}  This paper concentrates on the well-posedness and evaluation of $\lambda$.  
  The interpretation of $\lambda$, and especially of the sign of $\lambda$, should involve statements about the behavior of the original non-linear system \eqref{sys}.  We conjecture that if $\lambda < 0$ then all solutions with $(\varphi(0),\dot{\varphi}(0)) \neq (\pi,0)$ converge to the single mode solution, and that if $\lambda > 0$ then all solutions with $(\varphi(0),\dot{\varphi}(0)) \not\in \{(0,0),(\pi,0)\}$ converge to a stationary solution with values in $\R^2 \times \bigl((\R/(2 \pi \Z) \times \R) \setminus \{(0,0),(\pi,0)\}\bigr)$.  In particular, if $\lambda <0$ then the single mode and the unstable single mode solutions are the only stationary solutions of \eqref{sys}, while if $\lambda >0$ then there is also a stationary solution satisfying $(\varphi(t),\dot{\varphi}(t)) \not\in \{(0,0),(\pi,0)\}$ for all $t \ge 0$. 
   
\end{remark}



\section{A more general linear system} \label{sec more gen}
Much of our analysis and computation of the Lyapunov exponent \eqref{E:lambda2} is valid in a more general setting.   Notice that \eqref{phi eta} can be rewritten as 
      \begin{equation} \label{phi eta2}
      \ddot{\varphi}(t) + 2 \zeta_2 \dot{\varphi}(t) + \bigl(\kappa^2 - \ddot{\eta}(t)\bigr) \varphi(t)  = 0. 
    \end{equation} 
with $\eta(t)$ given by \eqref{eta}. We will replace the generalized process $\ddot{\eta}(t)$ by a more general $\xi(t)$.  

Let $\{v(t): t \ge 0\}$ be the $\R^d$ valued Ornstein-Uhlenbeck process
given by 
     \begin{equation} \label{v}
     dv(t) = Av(t) dt +  B dW(t) 
    \end{equation}
where $A$ is a $d \times d$ matrix, and $B$ is a $d \times m$ matrix and $W(t)$ is a standard $m$-dimensional Brownian motion.  See for example Gardiner \cite[Sect 4.4]{Gard} or Liberzon and Brockett \cite{LibBro}.  We shall assume throughout that the eigenvalues of $A$ have strictly negative real parts, and that $(A,B)$ is a controllable pair, that is 
\begin{equation} \label{CP}
    \mbox{rank}\{B,AB, \ldots, A^{d-1}B\} = d.
   \end{equation}
   
Since the eigenvalues of $A$ have strictly negative real parts, the equation \eqref{v} has a stationary solution 
    $$
    v(t) = \int_{-\infty}^t e^{(t-s)A} BdW(s)
    $$
with a stationary probability measure $\mu$, say, which is mean-zero Gaussian with $d\times d$ covariance matrix $R$ given by  \begin{equation} \label{R}
    R_{jk}  = \sum_{\ell = 1}^m \int_{-\infty}^0 \bigl(e^{-sA}Be_\ell\bigr)_j\bigl(e^{-sA}Be_\ell\bigr)_k ds = \int_0^\infty \bigl(e^{tA}BB^\ast e^{tA^\ast}\bigr)_{jk}dt.
    \end{equation}    
The controllable pair assumption is then equivalent to the assumption that $R$ is invertible, and this in turn is equivalent to the assumption that $\mbox{supp}(\mu) = \R^d$.

       Choose $a \in\R^d$ and $\gamma \in \R^m$ and define
   \begin{equation} \label{xi}
   \xi(t) = \sum_{j=1}^d a_jv_j(t) + \sum_{\ell = 1}^m \gamma_\ell \dot{W}_\ell(t) = \langle a,v(t) \rangle  + \langle \gamma, \dot{W}(t)\rangle.
   \end{equation}
At the cost of replacing $m$ with $m+1$ and adjoining a column of zeroes to the matrix $B$, we can include the situation where $\xi(t)$ includes some white noise which is independent of the white noise driving $v(t)$.  With this choice of $\xi(t)$  the equation
     \begin{equation} \label{phi xi2}
      \ddot{\varphi}(t) + 2\zeta_2 \dot{\varphi}(t) + \bigl(\kappa^2 - \xi(t)\bigr) \varphi(t)  = 0 
    \end{equation}    
 has a well-defined meaning as the second order SDE
     \begin{equation} \label{phi xi}
      \ddot{\varphi}(t) + 2 \zeta_2 \dot{\varphi}(t) + \bigl(\kappa^2 - \langle a,v(t)\bigr) \varphi(t)  =    \sum_{\ell = 1}^m \gamma_\ell \varphi(t)\dot{W}_\ell(t). 
    \end{equation}   
We will consider the Lyapunov exponent 
   \begin{equation}
\label{lam gen}
\lambda :=\lim \limits_{t\to\infty}\frac{1}{t}\log\|\left(\varphi(t), \dot\varphi(t) \right)\|.
\end{equation}
for the more general system given by (\ref{v},\ref{phi xi}).

\s

Since $v(t) = \begin{pmatrix} \eta(t) \\ \dot{\eta}(t)\end{pmatrix}$ satisfies 
    \begin{equation}  \label{v eta}
  dv(t)   = 
  \begin{pmatrix} 0 & 1 \\ -\chi^2 & -2 \zeta_1  \end{pmatrix}  v(t) dt 
    + \begin{pmatrix} 0 \\ \nu\end{pmatrix}dW(t). 
                      \end{equation} 
    and 
  $$
  \ddot{\eta}(t) = - \chi^2 \eta(t) -2 \zeta_1 \dot{\eta}(t)  + \nu \dot{W}(t) 
  $$
we see that the system (\ref{eta},\ref{phi eta}) is a special case of (\ref{v},\ref{phi xi}).  

\section{Khas'minskii's integral formula} \label{sec Khas}

Define $u(t) = \begin{pmatrix} \varphi(t) \\ \dot{\varphi}(t) \end{pmatrix}$, then \eqref{phi xi} can be written as the 2-dimensional linear SDE
     \begin{equation}  \label{uu xi}
     du(t)
        = \begin{pmatrix} 0 & 1 \\[1ex] -\kappa^2 +\langle a,v(t)\rangle & -2 \zeta_2 \end{pmatrix} u(t)  dt  +\sum_{\ell = 1}^m  \begin{pmatrix} 0 & 0 \\[1ex] \gamma_\ell & 0\end{pmatrix}u(t) dW_\ell(t).
               \end{equation} 
We note that the Stratonovich and It\^o interpretations of \eqref{uu xi} are the same.   Following Khas'minskii \cite{Khas67}, write $u(t) = \|u(t)\|\begin{pmatrix} \cos \psi(t) \\ \sin \psi(t) \end{pmatrix}$.   Applying It\^{o}'s formula to \eqref{uu xi} gives
  $$
  d(\log \|u(t)\|) =  Q(v(t),\psi(t))dt + \sum_{\ell = 1}^m \gamma_\ell \sin \psi(t) \cos \psi(t) dW_\ell(t)
  $$
and 
  \begin{equation} \label{psi}
   d \psi(t) = h(v(t),\psi(t))dt +\sum_{\ell = 1}^m \gamma_\ell \cos^2 \psi(t) dW_\ell(t)
   \end{equation}
where
  \begin{equation} \label{Q}
  Q(v,\psi) =  \bigl(1-\kappa^2+\langle a,v\rangle\bigr) \sin\psi \cos \psi -2\zeta_2 \sin^2 \psi +\frac{\|\gamma\|^2}{2} \cos^2 \psi \cos 2\psi
  \end{equation}
and
  $$
   h(v,\psi) = -1+\bigl(1-\kappa^2 +\langle a,v\rangle\bigr) \cos^2 \psi - 2 \zeta_2 \sin \psi \cos \psi -  \|\gamma\|^2 \sin \psi \cos^3 \psi.
   $$                
The process $\{(v(t),\psi(t)): t \ge 0\}$ given by (\ref{v},\ref{psi}) is a diffusion on $M:= \R^d \times \R/(2 \pi \Z)$ with generator ${\cal A}$, say.  For $(\varphi(0),\dot{\varphi}(0)) \neq (0,0)$ we have
    \begin{align}
    \lim_{t \to \infty} \frac{1}{t} \log \sqrt{\varphi^2(t) + \dot{\varphi}^2(t)} \nonumber
     & = \lim_{t \to \infty} \frac{1}{t} \log \|u(t)\| \\ 
       & = \lim_{t \to \infty} \frac{1}{t} \left[ \int_0^t Q(v(s),\psi(s))ds + \sum_{\ell = 1}^m \int_0^t \frac{\gamma_\ell}{2} \sin 2\psi(s)\, dW_\ell(s) \right] \label{int}
    \end{align} 
in the sense that the almost-sure existence and value of the limit on the left under initial conditions $(v,\varphi,\dot{\varphi})$ with $(\varphi,\dot{\varphi}) \neq (0,0)$ is equivalent to the almost-sure existence and value of the limit on the right under initial conditions $(v,\psi)$.

Since
    $$
    N(t) =  \sum_{\ell = 1}^m \int_0^t \frac{\gamma_\ell}{2} \sin 2 \psi(s)\, dW_\ell(s) 
    $$
is a continuous martingale with quadratic variation $\langle N\rangle_t \le \|\gamma\|^2 t /4$, it follows that 
   \begin{equation} \label{int2}
   \PP^{(v,\psi)}\left( \frac{1}{t}N(t) \to 0 \mbox{ as }t \to \infty \right) = 1
   \end{equation}
for all $(v,\psi) \in  M$.   

\begin{proposition} \label{prop ergodic} Assume that the eigenvalues of $A$ have strictly negative real parts, and that $(A,B)$ is a controllable pair.  Assume also that the coefficients $a \in \R^d$ and $\gamma \in\R^m$ in \eqref{xi} are not both zero.  Then the diffusion $\{(v(t),\psi(t)): t \ge 0\}$ on $M$ has a unique stationary probability measure $m$, say, with smooth density $\rho(v,\psi)$.  Moreover if $F: M \to \R$ is $m$ integrable then
    \begin{equation} \label{ergodic}
    \PP^{(v,\psi)}\left(\frac{1}{t}\int_0^t F(v(s),\psi(s))ds  \to \int_M F\,dm  \mbox{ as }t \to \infty \right) = 1
    \end{equation}
for every $(v,\psi) \in M.$
\end{proposition}

The proof of Proposition \ref{prop ergodic} is given in Section \ref{sec ergodic}.
   The $v$ marginal of $m$ is the stationary probability measure $\mu$ on $\R^d$ which is mean-zero Gaussian with mean zero and covariance matrix $R$ given by \eqref{R}.  There exist constants $k_0$ and $k_1$ such that $|Q(v,\psi)| \le k_0 + k_1\|v\|$ and so 
   $$
   \int_M |Q(v,\psi)|dm(v,\psi) \le \int_{\R^d}\bigl(k_0+k_1\|v\|\bigr)d\mu(v) <\infty.
   $$
By \eqref{int} and \eqref{int2} and using Proposition \ref{prop ergodic} we obtain

\begin{corollary} \label{cor ergodic}  (i) Under the assumptions of Proposition \ref{prop ergodic} we have 
   \begin{equation} \label{ergodic2}
    \PP^{(v,\varphi,\dot{\varphi})}\left(\frac{1}{t} \log \|(\varphi(t),\dot{\varphi}(t))\|  \to \int_M Q\, dm  \mbox{ as }t \to \infty \right) = 1
    \end{equation}
for every $(v,\varphi,\dot{\varphi})$ with $(\varphi,\dot{\varphi}) \neq (0,0)$.

(ii) The Lyapunov exponent $\lambda$ defined in \eqref{lam gen} exists as an almost-sure limit for all initial conditions $(v,\varphi,\dot{\varphi})$ with $(\varphi,\dot{\varphi}) \neq (0,0)$ and 
    \begin{equation} \label{lam Khas}
    \lambda = \int_M Q(v,\psi)\,dm(v,\psi).
    \end{equation}
 \end{corollary}

\subsection{Evaluation} Direct evaluation of \eqref{lam Khas} is hard, because we need to first solve ${\cal A}^\ast m = 0$.  There is a simple case when $a = 0$, because then $v(t)$ disappears from \eqref{uu xi} and problem reduces to a constant coefficient linear SDE in $\R^2$.  In this case the exact formula of Imkeller and Lederer \cite{IL99} will apply.  But this excludes the original block-pendulum model.  The block-pendulum case is interesting because of the different sorts of noise in \eqref{phi eta}:  white noise $\dot{W}(t)$ as well as colored noise $\eta(t)$ and $\dot{\eta}(t)$.

Before proceeding to the small noise estimates, we give a simple upper bound for the Lyapunov exponent $\lambda$.

\begin{proposition} \label{prop lam upper}  Assume $\zeta_2 < \kappa$ and let $\kappa_d = \sqrt{\kappa^2 - \zeta_2^2}$ denote the damped frequency of the pendulum.  Then
  \begin{equation} \label{lam upper xi}
    \lambda \le - \zeta_2 + \frac{\sqrt{\langle a, Ra \rangle}}{2 \kappa_d} + \frac{\|\gamma\|^2}{2 \kappa_d^2},
  \end{equation}
where $R$ is the covariance matrix \eqref{R}.  For the original block and pendulum setting
   \begin{equation} \label{lam upper eta}
    \lambda \le - \zeta_2 + \frac{\nu}{2 \kappa_d}\sqrt{\frac{\chi^2+4 \zeta_1^2}{4 \zeta_1}} + \frac{\nu^2}{2 \kappa_d^2}.
  \end{equation}
\end{proposition}

The proof of Proposition \ref{prop lam upper} uses a slightly different version of the function $Q(v,\psi)$ and is given in Section \ref{sec Khas rev}.  The result is motivated by an estimate in Ariaratnam \cite{Aria91}.

\section{Small forcing} \label{sec scaling}  
 
In this section we consider the effect of small forcing of the block. Precisely, we replace
   $
   \nu \rightsquigarrow \eps \nu 
   $
for small $\eps$.  The linearized system (\ref{eta},\ref{phi eta}) becomes
  \begin{align*}
    \ddot{\eta}(t) +2 \zeta_1 \dot{\eta}(t)  + \chi^2 \eta(t) &  = \eps \nu \dot{W}(t),\\
    \ddot{\varphi}(t)   + 2 \zeta_2 \dot{\varphi}(t)   + \big(\kappa^2 +2 \zeta_1 \dot{\eta}(t)   + \chi^2 \eta(t)\big) \varphi(t)  &   = \eps \nu \varphi(t)   \dot{W}(t).   
 \end{align*}
Equivalently, writing $\eta(t) = \eps \tilde{\eta}(t)$ and then dropping the tilde
   \begin{align}
     \ddot{\eta}(t) +2 \zeta_1 \dot{\eta}(t)  + \chi^2 \eta(t) &  = \nu \dot{W}(t) \label{eta eps},\\      \ddot{\varphi}(t)   + 2 \zeta_2 \dot{\varphi}(t)   + \big(\kappa^2 +2 \eps \zeta_1 \dot{\eta}(t)   + \eps\chi^2 \eta(t)\big) \varphi(t)  &   = \eps \nu \varphi(t)   \dot{W}(t). \label{phi eta eps}
  \end{align}
For the more general problem in Section \ref{sec more gen} we replace
  $
   \xi(t) \rightsquigarrow \eps \xi(t) $ 
in \eqref{phi xi2}, that is $a \rightsquigarrow  \eps a$ and $\gamma \rightsquigarrow  \eps \gamma$, giving the equation 
   \begin{equation}
    \ddot{\varphi}(t)   + 2 \zeta_2 \dot{\varphi}(t)   + \big(\kappa^2 - \eps \langle a,v(t) \rangle \big) \varphi(t)     = \sum_{\ell = 1}^m \eps \gamma_\ell \varphi(t)   \dot{W}_\ell(t)  \label{phi xi eps}
 \end{equation}
   where $\{v(t): t \ge 0\}$ is still given by \eqref{v}. 
 We now define the Lyapunov exponent
\begin{equation}  \label{lam eps}
\lambda(\eps) :=\lim \limits_{t\to\infty}\frac{1}{t}\log\|\left(\varphi(t), \dot\varphi(t) \right)\|.
\end{equation}
where $\{\varphi(t): t \ge 0\}$ is given by \eqref{phi xi eps}.

\s

Define the matrix valued cosine transform of $A$ by 
   \begin{equation} \label{hatS}
   \widehat{S}_A(\omega) = \frac{1}{\pi}\int_0^\infty e^{tA} \cos \omega t \,dt,
   \end{equation}
and recall that $R$ denotes the covariance matrix of the stationary probability measure $\mu$ for $\{v(t): t \ge 0\}$.

\begin{theorem}  \label{thm lam}  (i) Assume that the eigenvalues of $A$ have strictly negative real parts, and that $(A,B)$ is a controllable pair.  Assume $\zeta_2 < \kappa$ and let $\kappa_d = \sqrt{\kappa^2 - \zeta_2^2}$ denote the damped frequency of the pendulum.  For $\{\varphi(t): t \ge 0\}$ given by \eqref{phi xi eps} we have 
    \begin{equation}  \label{lam eps}
    \lambda(\eps) = -\zeta_2 + \eps^2\lambda_2(2\kappa_d) + O(\eps^4)\quad \mbox{ as } \eps \to 0.
     \end{equation}
where
    \begin{equation} \label{lam2 omega}
    \lambda_2(\omega) =\frac{\pi}{\omega^2} \Bigl( \langle a, \widehat{S}_A(\omega)Ra\rangle + \langle a, \widehat{S}_A(\omega)B\gamma \rangle + \frac{\|\gamma\|^2}{2\pi } \Bigr).
    \end{equation}
    
(ii) Moreover, given matrices $A$ and $B$ and vectors $a$ and $\gamma$ the asymptotic \eqref{lam eps} is uniform for $\kappa_d$ bounded away from 0 and $\infty$.  That is, given $A$, $B$, $a$, $\gamma$ and $0 < c_1 < c_2 <\infty$ there exists $K$ such that
     \begin{equation}  \label{lam eps unif}
   \Bigl| \lambda(\eps) +\zeta_2 - \eps^2\lambda_2(2\kappa_d)\Bigr| \le K\eps^4     
     \end{equation} 
whenever $0 < \eps \le 1$ and $c_1 \le \kappa_d \le c_2$. 
     \end{theorem}

The proof for Theorem \ref{thm lam} is given in Section \ref{sec lamproof}.

\s

Observe that the destabilizing effect of the noise is strongest when the damped frequency $\kappa_d$ maximises the function $\lambda_2(2\omega)$.  We will expand on this observation in Section \ref{sec scaling2}.  

\subsection{Evaluation of $\lambda_2(\omega)$}

In the special case when $a = 0$ we have $\lambda_2(\omega) = \|\gamma\|^2/(2 \omega^2)$ and we recover the result of Auslender and Milstein \cite{AM82} that
     $$
     \lambda(\eps) = -\zeta_2+  \frac{\eps^2\|\gamma\|^2}{8 \kappa_d^2} + O(\eps^4) \quad \mbox{ as } \eps \to 0.
     $$
But when $a \neq 0$ we have to do some work to simplify the formula \eqref{lam2 omega} for $\lambda_2(\omega)$.     
   The covariance matrix $R$ is determined by \eqref{R}.  The cosine transform $\widehat{S}_A(\omega)$ defined in \eqref{hatS} satisfies  
  \begin{equation} \label{hatS form}
   \widehat{S}_A(\omega)  =  \frac{1}{\pi} \Re\Bigl( \int_0^\infty e^{tA} e^{-i\omega t}\,dt\Bigr)
   = - \frac{1}{\pi} \Re\Bigl( (A-i\omega I_d)^{-1}\Bigr).
  \end{equation}
Hence, given the matrices $A$ and $B$ and the vectors $a$ and $\gamma$, all the terms in \eqref{lam2 omega} can be calculated.  

We will give formulas for $\lambda_2(\omega)$ in terms of power spectral density.  Recall that for any $L^2$ stationary mean zero scalar process $X(t)$ the power spectral density function $S_X$ is given by
    \begin{equation} \label{psd}
    S_X(\omega) = \frac{1}{2\pi} \int_{-\infty}^\infty \E[X(t)X(0)] e^{-i \omega t}\,dt = \frac{1}{\pi} \int_0^\infty \E[X(t)X(0)] \cos \omega t\,dt.
    \end{equation}                
We note that many authors omit the factor $1/2\pi$.

\s

The block and pendulum setting has $\xi(t) = \ddot{\eta}(t)$ where $\dot{\eta}(t)$ is a well-defined $L^2$ stationary process.  The first result extends this case to the general system of Section \ref{sec more gen}.

 
Fix $\alpha \in \R^d$ and define $\xi(t) = \dot{V}(t)$ where $V(t) = \langle \alpha,v(t)\rangle $. 
 Since 
    \begin{equation} \label{Vdot}
     \xi(t) = \langle \alpha,\dot{v}(t) \rangle = \langle \alpha, Av(t) \rangle + \langle \alpha, B\dot{W}(t) \rangle  = \langle A^\ast \alpha, v(t) \rangle + \langle B^\ast \alpha, \dot{W}(t) \rangle,
    \end{equation} 
we have \eqref{xi} with $a = A^\ast \alpha$ and $\gamma = B^\ast \alpha$.   Note that $V(t)$ is a well-defined $L^2$ stationary process, so that it has a well-defined autocovariance function $\E[V(t)V(0)]$ and hence a well-defined power spectral density $S_V(\omega)$ using the formula \eqref{psd}. 

 \begin{proposition}  \label{prop Vdot}  For $\xi$ given by \ref{Vdot} we have  
    \begin{equation}  \label{lam2 Vdot}
     \lambda_2(\omega) = \pi S_V(\omega).    
     \end{equation}
 \end{proposition}

Specializing to the original block and pendulum case, we have $\xi(t) = \ddot{\eta}(t) = \dot{v_2}(t)$ where  $v(t) = \begin{pmatrix}\eta(t) \\ \dot{\eta}(t) \end{pmatrix}$ as in \eqref{v eta}.  Therefore we may apply Proposition \ref{prop Vdot} with $V(t) = v_2(t) = \dot{\eta}(t)$.  For the system \eqref{v eta} the power spectral density $S_{\dot{\eta}}(w)$ is well known, see for example \cite[Sect.4.4]{Gard}. 
 \begin{corollary}  \label{cor block}  For $\xi =\ddot{\eta}$ given by \ref{eta} we have 
    \begin{equation}  \label{lam2 block}
     \lambda_2(\omega) = \pi S_{\dot{\eta}}(\omega)  = \frac{\omega^2 \nu^2}{2[(\chi^2-\omega)^2+4 \zeta_1^2 \omega^2]}.    
     \end{equation}
 \end{corollary}
 
\s

In the general case \eqref{xi} when $\xi(t) = \langle a,v(t)\rangle + \langle \gamma, \dot{W}(t)\rangle $ with $\gamma \neq 0$ then $\xi$ is not a well-defined $L^2$ process.  It does not have an autocovariance function, and the formula \eqref{psd} for the power spectral density does not apply.  However we can consider the power spectral density $S_{\xi_\delta}(\omega)$ for a mollified version $\xi_\delta$ of $\xi$. 
 
Suppose $\psi:\R \to [0,\infty)$ is piecewise continuous with support in $[-1,1]$ and \\$\int_{-\infty}^\infty \psi(t)dt = 1$.  Define $\psi_\delta(t) = (1/\delta)\psi(t/\delta)$.  For any continuous function $f$ we have 
    \begin{equation} \label{fd}
  f_\delta(t) := f \ast \psi_\delta(t) =  \int_{-\infty}^\infty f(s)\psi_\delta(t-s) ds = \int_{-\infty}^\infty f(t-\delta u)\psi(u)du \to f(t)
  \end{equation}
as $\delta  \to 0$, so that the functions $\psi_\delta(t-\cdot)$ converge to the Dirac delta distribution at $t$. 
 Then the process $\xi_\delta$ defined informally as $\xi*\psi_\delta$ and precisely by
   \begin{equation} \label{xi delta}
    \xi_\delta(t) = \int_{-\infty}^\infty \langle a,v(s)\rangle \psi_\delta(t-s)ds + \left\langle \gamma, \int_{-\infty}^\infty \psi_\delta(t-s)dW(s) \right\rangle
    \end{equation}
is well-defined $L^2$ process, and it converges in a weak sense to the generalized process $\xi$.

\begin{proposition} \label{prop psd xi}   For $\xi(t) = \langle a,v(t)\rangle + \langle \gamma, \dot{W}(t)\rangle $ and $\xi_\delta(t)$ given by \eqref{xi delta} the limit
       \begin{equation} \label{psd lim}
  \lim_{\delta \to 0} S_{\xi_\delta}(\omega) = \langle a,\widehat{S}_A(\omega) Ra\rangle+ \langle a, \widehat{S}_A(\omega)B\gamma \rangle + \frac{\|\gamma\|^2}{2\pi}
    \end{equation}
exists and does not depend on the choice of $\psi$.  Defining $S_\xi(\omega):= \lim _{\delta \to 0} S_{\xi_\delta}(\omega)$ we have 
   \begin{equation}  \label{lam2 xi}
     \lambda_2(\omega) = \frac{\pi}{\omega^2} S_\xi(\omega).    
     \end{equation}      
  \end{proposition}

The proofs of Propositions \ref{prop Vdot} and \ref{prop psd xi} are given in Section \ref{sec lamproof}.

\section{Small forcing and small pendulum damping}   \label{sec scaling2}
 
 In this section we consider the combined effect of small forcing of the block together with small damping of the pendulum.  Precisely, we replace
   $$
   \nu \rightsquigarrow \eps \nu \quad \mbox{ and } \quad \zeta_2 \rightsquigarrow \eps^2 \zeta_2
   $$
for small $\eps$.  The equation for $\varphi$ is now
   \begin{equation}
     \ddot{\varphi}(t)   + 2 \eps^2 \zeta_2 \dot{\varphi}(t)   + \big(\kappa^2 +2 \eps \zeta_1 \dot{\eta}(t)   + \eps \chi^2 \eta(t)\big) \varphi(t)  =    \eps \nu \varphi(t)   \dot{W}(t) \label{phi eta eps2}
  \end{equation}
for the original block and pendulum setting, and 
   \begin{equation}
    \ddot{\varphi}(t)   + 2 \eps^2 \zeta_2 \dot{\varphi}(t)   + \big(\kappa^2 - \eps \langle a,v(t) \rangle \big) \varphi(t)     = \sum_{\ell = 1}^m \eps \gamma_\ell \varphi(t)   \dot{W}_\ell(t)  \label{phi xi eps2}
 \end{equation}
in the general setting.  In order to distinguish this scaling from the previous one in Section \ref{sec scaling} we write
\begin{equation}  \label{lamhat eps}
\widehat{\lambda}(\eps) :=\lim \limits_{t\to\infty}\frac{1}{t}\log\|\left(\varphi(t), \dot\varphi(t) \right)\|.
\end{equation}
for the Lyapunov exponent where $\{\varphi(t): t \ge 0\}$ is given by \eqref{phi xi eps2}.

\begin{theorem}  \label{thm lam2}  Assume that the eigenvalues of $A$ have strictly negative real parts, and that $(A,B)$ is a controllable pair.  For the process $\big\{\big(\varphi(t), \dot{\varphi}(t) \big): t\geq 0  \big\}$ given by \eqref{phi xi eps2} we have 
    \begin{equation}  \label{lamhat eps2}
    \widehat{\lambda}(\eps) = \eps^2\bigl(-\zeta_2 + \lambda_2(2\kappa)\bigr) + O(\eps^4)\quad \mbox{ as } \eps \to 0.
     \end{equation}
where $\lambda_2(\omega)$ is given by \eqref{lam2 omega}.
     \end{theorem}

\begin{proof}  The damped frequency for \eqref{phi xi eps2} is $\kappa_d = \sqrt{\kappa^2- \eps^4 \zeta_2^4} = \kappa + O(\eps^4)$, and so $\lambda_2(2\kappa_d) = \lambda_2(2\kappa) + O(\eps^4)$.  The result now follows directly from Theorem \ref{thm lam}(ii).  \end{proof}

\s

Returning to the original block and pendulum case and using \eqref{lam2 block} we have 
     \begin{equation}  \label{lam eps2}
    \widehat{\lambda}(\eps) = \eps^2\left(-\zeta_2 + \frac{2 \nu^2 \kappa^2} {(\chi^2-4 \kappa^2)^2+16 \zeta_1^2 \kappa^2}\right) + O(\eps^4)\quad \mbox{ as } \eps \to 0.
     \end{equation}
Recall that the long term exponential growth or decay rate for the process $\big\{\big(\varphi(t), \dot{\varphi}(t) \big): t\geq 0  \big\}$ is given by the Lyapunov exponent $\widehat{\lambda}(\eps)$.   
 Putting $\widehat{\lambda} (\eps)=0$ in \eqref{lam eps2} provides the almost-sure {\em stability boundary} in terms of the excitation intensity $\nu$.  
 Hence the second order approximation of the almost-sure stability boundary is given by
\begin{equation} \label{nu crit}
\nu_{c}^2
= \frac{ \zeta_2 \big[(\chi^2-4\kappa^2)^2 + 16\zeta_1^2 \kappa^2 \big] }{ 2\kappa^2 } 
= \zeta_2\left(\frac{(\chi^2-4\kappa^2 )^2}{2\kappa^2} + 8\zeta_1^2\right) .
\end{equation}
 \begin{figure}[h]
   \begin{center}
   \includegraphics[scale=0.8]
{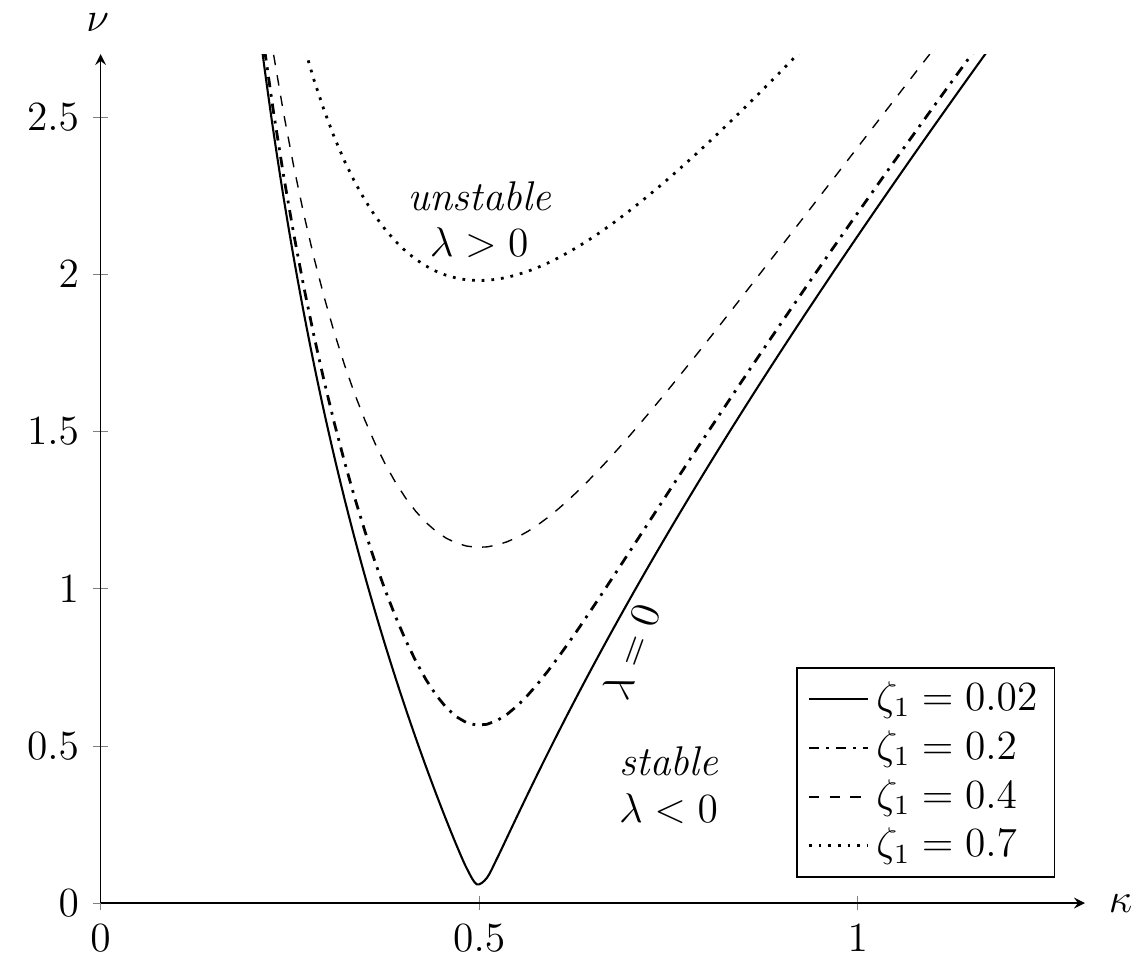}
   \end{center}
  \caption{Almost-sure stability boundaries with $\chi = 1$ and $\zeta_2 = 1$ and $\zeta_1 = 0.02$, 0.2, 0.4 and 0.7.}
   \label{fig nu crit}  
  \end{figure}  
It is clear that the dissipation in both the primary ($\zeta_1$) and secondary ($\zeta_2$) systems has a stabilizing effect on the single mode solution.  Although  no particular attention was given to the $1:2$ resonance in the analysis of the linearized system, the stability boundary \eqref{nu crit} shows the significance of internal resonance,  $\chi  \approx 2\kappa$, in determining the instability region in the $(\kappa,\nu)$ parameter space, which is of significance in applications.
  For fixed $\chi$, $\zeta_1$ and $\zeta_2$ the critical noise intensity $\nu_c$, as a function of $\kappa$, has a minumum value $\sqrt{8 \zeta_1^2 \zeta_2}$ attained when $\kappa = \chi/2$.  Figure \ref{fig nu crit} shows $\nu_c$ as a function of $\kappa$ with $\chi = 1$ and $\zeta_2 = 1$ and $\zeta_1 = 0.02$, 0.2, 0.4 and 0.7.
 
 \s

The behavior displayed in Figure \ref{fig nu crit} mimics that of the instability tongues and transition curves in the stability chart for Mathieu's equation with linear viscous damping and cosine periodic forcing.   More specifically the equation
    \begin{equation} \label{math}
      \ddot{\varphi}(t) + 2 \eps \zeta_2 \dot{\varphi}(t) +(\kappa^2  - \eps \nu \cos \omega t) \varphi(t) =   0 
    \end{equation}  
with small periodic forcing $\eps \nu$ and small dissipation $\eps \zeta_2$ has first order approximation (as $\eps \to 0$) of the stability boundary given by 
    \begin{equation} \label{math nu crit}
    \eps^2(\nu_c^2-16\zeta_2^2) = \omega^2 \left(\frac{4 \kappa^2}{\omega^2} - 1\right)^2.
   \end{equation}
See Verhulst \cite[page 241]{Ver02} for the case $\omega = 2$.   Note that in the deterministic setting the damping is of the same order $\eps$ as the forcing.   
Figure \ref{fig mathieu} shows $\nu_c$ as a function of $\kappa$ with $\omega = 1$, $\zeta_2 = 0.1$ and $\eps = 0.2$, 0.1 and 0.05.  Notice that in this model the ``width'' of the instability region decreases as $\eps$ decreases. 
\begin{figure}[h] 
   \begin{center}
   \includegraphics[scale=0.8]{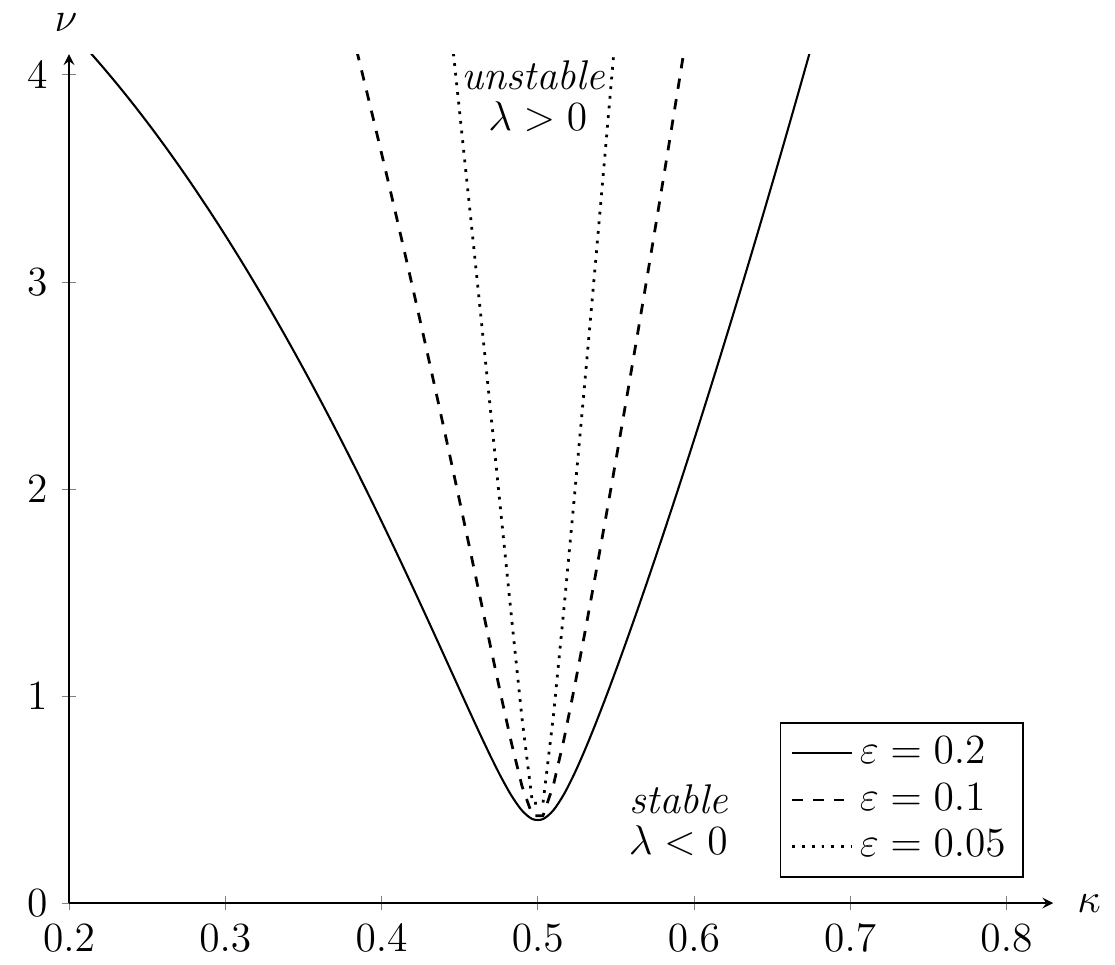}
   \end{center}
   \caption{Almost-sure stability boundaries for the Mathieu equation \eqref{math} with $\omega = 1$, $\zeta_2 = 0.1$ and $\eps = 0.2$, 0.1 and 0.05.}\label{fig mathieu}
  \end{figure}
  
A more directly relevant comparison is seen when we replace white noise forcing in the autoparametric system with periodic (deterministic) forcing.  Consider the system   
\begin{equation}
\label{sys cos}
  \begin{split} 
\ddot{\eta}(t) + 2 \zeta_1 \dot {\eta}(t) + \chi^2 \eta(t) - R \big(\ddot{\theta}(t) \sin{\theta}(t)  + {\dot{\theta}}^2(t) \cos{\theta}(t)\big) &= \eps \nu \cos \omega t\\
 \ddot{\theta}(t) +  2 \eps\zeta_2  \dot {\theta}(t) +  \big(\kappa^2  - \ddot{\eta}(t)  \big)\sin{\theta}(t) &= 0,
\end{split}
\end{equation}  
with forcing of some fixed frequency $\omega$ and small intensity $\eps\nu$, and small pendulum damping $\eps \zeta_2$.  Linearizing along the single mode solution $\theta(t) \equiv 0$ we get
    \begin{align}
     \ddot{\eta}(t) +2 \zeta_1 \dot{\eta}(t)  + \chi^2 \eta(t) &  = \eps\nu \cos \omega t \label{eta cos}\\
     \ddot{\varphi}(t)   + 2 \eps \zeta_2 \dot{\varphi}(t)   + \bigl(\kappa^2 - \ddot{\eta}(t)\bigr) \varphi(t)  &   = 0. \label{phi eta cos}  
 \end{align}
The stationary solution of \eqref{eta cos} is 
    $$
    \eta(t) = \frac{\eps \nu}{\sqrt{(\chi^2-\omega^2)^2+4\zeta_1^2 \omega^2}} \cos(\omega t +\alpha)
    $$
where $\alpha = \mbox{arg}(\chi^2-\omega^2+2i\zeta_1 \omega)$,
and so 
  $$
    \ddot{\eta}(t) = \frac{\eps \omega^2 \nu}{\sqrt{(\chi^2-\omega^2)^2+4\zeta_1^2 \omega^2}} \cos(\omega t +\alpha+\pi).
    $$ 
Therefore \eqref{phi eta cos} has the form of Mathieu's equation \eqref{math} with $\nu$ replaced by $ \dfrac{\omega^2 \nu}{\sqrt{(\chi^2-\omega^2)^2+4\zeta_1^2 \omega^2}}$.  The phase change $\pi+\alpha$ has no effect on the stability, and the (first order) stability boundary is    
\begin{equation} \label{instab bp}
     \eps^2\left( \frac{\omega^4 \nu_c^2}{(\chi^2-\omega^2)^2+4\zeta_1^2 \omega^2} - 16 \zeta_2^2\right) = \omega^2\left(\frac{4 \kappa^2}{\omega^2} - 1\right)^2.
   \end{equation}   
A multiplicative change in the vertical coordinate will convert the stability regions for the Mathieu equation \eqref{math} shown in Figure \ref{fig mathieu} into the corresponding regions for the periodically forced autoparametric system (\ref{eta cos},\ref{phi eta cos}).

\begin{remark} \label{rem eta ddot eta} For the periodically forced system (\ref{eta cos},\ref{phi eta cos}) the functions $\eta(t)$ and $\ddot{\eta}(t)$ are related by a simple multiplicative factor. It makes little theoretical difference whether $\nu \cos \omega t$ is applied as forcing on the block, or is assumed to describe the motion of the pivot point.  The situation with stochastic forcing is very different.  With white noise forcing $\{\eta(t): t \ge 0\}$ is an $L^2$ process with continous sample paths, whereas $\{\ddot{\eta}(t): t \ge 0\}$ exists only as a generalized process and has to be interpreted in terms of stochastic integrals.
\end{remark}

\section{Proof of Theorem \ref{thm well-posed}} \label{sec wellposed}

\subsection{Construction of a Lyapunov function}

Up to normalization, the energy of the block and pendulum system is given by 
    \begin{equation} \label{E}
    E(v_1,v_1,u_1,u_2) = \frac{1}{2}v_2^2 + \frac{1}{2} R u_2^2 - R  v_2 u_2 \sin u_1 + \frac{1}{2}\chi^2 v_1^2 + R \kappa^2(1-\cos u_1).
  \end{equation} 
  
   \begin{lemma} \label{lem Lest}
    $$
   {\cal L }E(v_1,v_1,u_1,u_2) = -2\zeta_1 v_2^2 -2R \zeta_2 u_2^2 + \frac{\nu^2}{2(1-R \sin^2 u_1)}.
   $$
  $$
  {\cal L}\Bigl( v_1(v_2-R u_2 \sin u_1)\Bigr) = v_2^2-R v_2 u_2 \sin u_1 - 2 \zeta_1 v_1v_2 - \chi^2 v_1^2.
  $$
\end{lemma}  

\begin{proof}  This is direct calculation.  \end{proof}

For $\alpha$ to be chosen later, define
  \begin{equation} \label{F}
   F(v_1,v_2,u_1,u_2)   = E(v_1,v_1,u_1,u_2) + \alpha \Bigl( v_1(v_2-R u_2 \sin u_1)\Bigr).
  \end{equation}
     
\begin{proposition} \label{prop F} (i) There exist $\alpha > 0$ and positive constants $c_1,\ldots,c_5$ such that 
  \begin{equation} \label{F ineq}
  c_1\|(v_1,v_2,u_2)\|^2 \le F(v_1,v_2,u_1,u_2) \le c_2+ c_3\|(v_1,v_2,u_2)\|^2
 \end{equation}
and 
  \begin{equation} \label{LF ineq}
  {\cal L}F(v_1,v_2,u_1,u_2) \le c_4-c_5 \|(v_1,v_2,u_2)\|^2.
  \end{equation}
(ii)  For $\alpha > 0$ as in (i), there exists $\beta_0 > 0$ such that for all $\gamma > 0$ there exists $c_6$ such that 
  \begin{equation} \label{LebF}
  {\cal L}(e^{\beta_0 F(v_1,v_2,u_1,u_2)}) \le - \gamma e^{\beta_0 F(v_1,v_2,u_1,u_2)} + c_6.
  \end{equation}
  \end{proposition}

\begin{proof}   (i) Several applications of the Cauchy-Schwarz inequality give
 \begin{align*}
 & \frac{1-\sqrt{R}}{2}(v_2^2+R u_2^2) + \frac{1}{2}\chi^2v_1^2 - \alpha \Bigl(v_1^2 + \frac{1}{2}v_2^2+ \frac{1}{2}R^2 u_2^2 \Bigr) \\
 &  \hspace{10ex} \le  F(v_1,v_2,u_1,u_2) \\
  & \hspace{10ex} \le R\kappa^2 +\frac{1+\sqrt{R}}{2}(v_2^2+Ru_2^2) + \frac{1}{2}\chi^2v_1^2 +  \alpha \Bigl(v_1^2 + \frac{1}{2}v_2^2+ \frac{1}{2}R^2 u_2^2 \Bigr).
  \end{align*}
Thus the upper and lower bounds on $F$ are satisfied whenever $0 <\alpha < \min(1-\sqrt{R},\chi^2/2)$.

Similarly, using Lemma \ref{lem Lest}, 
\begin{align*}
   {\cal L}F(v_1,v_2,u_1,u_2) & \le \frac{\nu^2}{2(1-R)} -2\zeta_1 v_2^2 -2R \zeta_2 u_2^2  
   + \alpha \Bigl(v_2^2+R| v_2 u_2|+ 2 \zeta_1| v_1v_2| - \chi^2 v_1^2 \Bigr)\\
   & \le  \frac{\nu^2}{2(1-R)} -2\zeta_1 v_2^2 -2R \zeta_2 u_2^2  
   + \frac{\alpha}{2} \Bigl((2+R+\frac{4\zeta_1^2}{\chi^2})v_2^2 + Ru_2^2 - \chi^2v_1^2 \Bigr).
   \end{align*} 
The upper bound on ${\cal L}F$ now follows by choosing $\alpha$ sufficiently small.  

(ii) Let $\Gamma$ denote the carr\'e du champ operator associated with ${\cal L}$, see for example \cite{BakEme85}.
 Since $ \partial F/dv_2$ and $\partial F/du_2$ both grow at most linearly with $\|(v_1,v_2,u_2)\|$ and the coefficients of $dW(t)$ in \eqref{sde} are bounded, there exists $c_7$ such that $\Gamma(F,F) \le c_7\|(v_1,v_2,u_2)\|^2$.  Then
    \begin{align*}
   {\cal L}\left(e^{\beta_0 F}\right) & = \Bigl(\beta_0 {\cal L}F + \beta_0^2 \Gamma(F,F)\Bigr)e^{\beta_0 F}\\
      & \le \Bigl(\beta_0(c_4-c_5 \|(v_1,v_2,u_2)\|^2)  + \beta_0^2 c_7 \|(v_1,v_2,u_2)\|^2\Bigr) e^{\beta_0 F} \\
     & =\Bigl(\beta_0 c_4-\beta_0(c_5-\beta_0 c_7)\|(v_1,v_2,u_2)\|^2 \Bigr) e^{\beta_0 F}.
  \end{align*}    
Choose $\beta_0 > 0$ so that $\beta_0 c_7 < c_5$.  For each $\gamma >0$ there a radius $r$ such that ${\cal L}(e^{\beta_0 F}) \le -\gamma e^{\beta_0 F}$ if $\|(v_1,v_2,u_2)\| \ge r$, and of course ${\cal L}(e^{\beta_0 F})$ is bounded in the set $\|(v_1,v_2,u_2)\| \le r$.  
\end{proof}

\subsection{Consequences of the Lyapunov function}

Recall the notation $(v_1,v_2,u_1,u_2)= U \in N$ and the abuse of notation $\|(v_1,v_2,u_2)\| = \|U\|$.  Fix $\alpha$ and $\beta$ so that the results of Proposition \ref{prop F} are valid for the function $F$.

\s

\n  {\it Proof of Theorem \ref{thm well-posed}(i)}. The coefficients of \eqref{sde} are locally Lipschitz, so there is a well-defined local solution.   Proposition \ref{prop F} implies $F(U) \to \infty$ as $\|U\| \to \infty$, and there exists $c$ such that ${\cal L}(1+F)(U) \le c(1+F)(U)$ for all $U$.  The result of Khas'minskii \cite[Thm 3.5]{Khas} implies that the well-defined local solution exists for all time and is a Feller process.

\n {\it Proof of Theorem \ref{thm well-posed}(ii)}. Notice that if $0 < c < c_5$ and $d = c_4+c$ and $b^2 = (c_4)/(c_5-c)$ then 
    $$
    {\cal L}F(U) \le c_4-c_5\|U\|^2 \le -c(1+\|U\|^2)+d1_{\|U\|\le b}
                       $$  
This gives condition (CD2) of Meyn and Tweedie \cite{MTIII}.  It now follows from \cite[Thm 4.3(i)]{MTIII} that $\PP(\|U_t\| \le b \mbox{ for some }t \ge \delta)= 1$ for all initial conditions $U(0)$ and all $\delta >0$.

\n {\it Proof of Theorem \ref{thm well-posed}(iii)}. The inequality \eqref{LebF} implies that $e^{\beta_0 F}$ satisfies condition (CD2) of \cite{MTIII}.  Taking $\beta = c_1 \beta_0$ and using \eqref{F ineq} for the first inequality, and then \cite[Th 4.3(ii)]{MTIII} for the second inequality gives
   $$
  \int_N \exp\bigl(\beta \|U\|^2\bigr)d\pi(U) \le \int_N \exp\bigl(\beta_0 F(U)\bigr)d\pi(U) \le \frac{c_6}{\gamma}
   $$  
for every invariant probability measure $\pi$ for \eqref{sde}.  

\n {\it Proof of Theorem \ref{thm well-posed}(iv)}. The function $e^{\beta_0 F}$ also satisfies condition (CD3) of \cite{MTIII}, so by the calculation in the proof of \cite[Thm 6.1]{MTIII} we have 
   $$
   \E[e^{\beta_0 F(U_t)}] \le e^{-\gamma t} e^{\beta_0 F(U)} + \frac{c_6}{\gamma}.
   $$
It remains only to use the upper and lower bounds on $F(U)$ given in \eqref{F ineq}.  We have $\beta = c_1 \beta _0$ and $\beta_1 = c_3 \beta_0$ and $K_1 = e^{c_2 \beta_0}$.  \qed

\section{Proof of Proposition \ref{prop ergodic}} \label{sec ergodic}

The generator ${\cal A}$ of the process $\{(v(t),\psi(t)): t \in \R\}$ on $M:= \R^d \times R/(2 \pi \Z)$ given by (\ref{v},\ref{psi}) can be written in the  H\"{o}rmander form
     $$
    {\cal A} =  \frac{1}{2} \sum_{\ell = 1}^m V_\ell^2+ V_0
  $$
where the vector fields $V_0, V_1, \ldots V_m$ are given by 
     $$
V_0(v,\psi) = \begin{pmatrix} Av \\[1ex] -1+(1-\kappa^2 + \langle a,v\rangle) \cos^2 \psi -  \zeta_2 \sin 2\psi   \end{pmatrix} \quad \mbox{ and } \quad  V_\ell(\theta,v) = \begin{pmatrix} Be_\ell \\[1ex] \gamma_\ell \cos^2 \psi  \end{pmatrix}
   $$
for $1 \le \ell \le m$.  

\subsection{Hypoellipticity}

Let $L = L(V_0,V_1,\ldots,V_m)$ be the Lie algebra generated by the vector fields $V_0,V_1,\ldots,V_m$ and let $L_0 = L_0(V_0;V_1,\ldots,V_m)$ be the ideal in $L$ generated by $V_1,\ldots,V_m$.

 For any smooth vector field $V$ on $M$, let ${\cal M}(V)$ denote the Lie bracket
    $${\cal M}(V) = [V,V_0],
    $$   
so that $-{\cal M}$ is the operation of taking the Lie derivative of a vector field with respect to $V_0$.

\begin{lemma} \label{lem calc}

(i) For all $k \ge 0$ and $1 \le \ell \le m$ 
 $$
   {\cal M}^k(V_\ell)(v,\psi) = \begin{pmatrix} A^kBe_\ell \\[1ex] f_{kl}(v,\psi) \end{pmatrix}
   $$
for some smooth function $f_{kl}: M \to \R$.

(ii) For a vector field of the form $V(v,\psi) = \begin{pmatrix} v_0 \\[1ex] 0 \end{pmatrix}$ we have 
  $$
  {\cal M}(V)(v,\psi) = \begin{pmatrix} Av_0 \\[1ex]   \langle a,v_0 \rangle \cos^2 \psi \end{pmatrix}.
  $$  
  
(iii)  For a vector field of the form $V(v,\psi) = \begin{pmatrix} v_0 \\[1ex] \alpha \cos^2 \psi \end{pmatrix}$ we have 
  $$
  [{\cal M}(V),V](v,\psi) = \begin{pmatrix} 0 \\[1ex]  2\alpha^2\cos^2 \psi \end{pmatrix} \quad \mbox{ and } \quad  [{\cal M}^2(V),{\cal M}(V)](v,\pm \pi/2) = \begin{pmatrix} 0 \\[1ex] 4\alpha^2  \end{pmatrix}.
  $$

\end{lemma}

\begin{proof} The calculations for (i) and (ii) are elementary and direct.  The calculations for (iii) are longer, but still elementary and direct.  We omit the details. 
\end{proof}

\begin{proposition} \label{prop liealg} Assume $a \in\R^d$ and $\gamma \in \R^m$ are not both 0.  Assume also that $(A,B)$ is a controllable pair.  Then ${\cal L}_0(v,\psi) = T_{(v,\psi)}M$ for all $(v,\psi) \in M$. 
   \end{proposition}

\begin{proof} If $\gamma \neq 0$ choose $\ell_0$ such that $\gamma_{\ell_0}\neq 0$ and consider the finite subset
    $$
     \{ {\cal M}^k(V_\ell): 0 \le k \le d-1,1 \le \ell \le m\}  \cup \{ [{\cal M}(V_{\ell_0}),V_{\ell_0}],\, [{\cal M}^2(V_{\ell_0}),{\cal M}(V_{\ell_0})]\} \equiv {\cal N}_1 \cup {\cal N}_2,
    $$    
say, of ${\cal L}_0$.  
By Lemma \ref{lem calc}(i) at each $(v,\psi)$ the $v$ components of the vector fields in ${\cal N}_1$ form the 
set $\{A^kBe_\ell: 0 \le k \le d-1,1 \le \ell \le m\}$ and this set spans $\R^d$ because of the controllable pair condition.  By Lemma \ref{lem calc}(iii) with $V = V_{\ell_0}$, at each $(v,\psi)$ the vector fields in ${\cal N}_2$ span $\{0\} \times \R$.  Together the set $({\cal N}_1 \cup {\cal N}_2)(v,\psi)$ spans $\R^d \times \R = T_{(v,\psi)}M$.  

If $\gamma = 0$ and $a \neq 0$ the controllable pair condition implies there exists $k_0 \ge 0$ and $1 \le \ell_0 \le m$ such that $\langle a,A^k Be_{\ell_0} \rangle = 0$ for $0 \le k <k_0$ and $\langle a,A^{k_0} Be_{\ell_0} \rangle \neq 0$.   By Lemma \ref{lem calc}(ii) we have ${\cal M}^k(V_{\ell_0})(v,\psi) = \begin{pmatrix}  A^kBe_{\ell_0} \\[1ex] 0 \end{pmatrix}$ for $k \le k_0$ and 
  $$
  {\cal M}^{k_0+1}(V_{\ell_0})(v,\psi) = \begin{pmatrix} A^{k_0+1}Be_{\ell_0}\\[1ex]   \langle a,A^{k_0}Be_{\ell_0} \rangle \cos^2 \psi\end{pmatrix}.
  $$
Since $\langle a,A^{k_0}Be_{\ell_0} \rangle \neq 0$ we can replace $V_{\ell_0}$ with ${\cal M}^{k_0+1}(V_{\ell_0})$ in ${\cal N}_2$ above.   We can apply Lemma \ref{lem calc}(iii) with $V = {\cal M}^{k_0+1}(V_{\ell_0})$, so that the vector fields in the new ${\cal N}_2$ span $ \{0\} \times \R$ and the proof is completed as before.  
\end{proof}

\subsection{Controllability}  

\begin{lemma}  \label{lem cont2}  Fix $t> 0$ and piecewise continuous $c:[0,t] \to \R$ with $c$ not identically zero.  For all $\psi_0$ and $\psi_1$ there exists $\alpha \in \R$ such that the path $\psi:[0,t] \to \R$ given by 
    \begin{equation} \label{psi c}
    \frac{d \psi(s)}{ds}= -1 +\bigl(1-\kappa^2 + \alpha c(s)\bigr)\cos^2 \psi(s)  - \zeta_2 \sin 2\psi(s), \quad 0 \le s \le t
    \end{equation}
with $\psi(0) = \psi_0$ has $\psi(t) = \psi_1$ \rm{mod} $2\pi$.
 \end{lemma} 

\begin{proof}  At the cost of changing the signs of $c$ and $\alpha$, we may assume there exists a subinterval $[t_0,t_1] \subset [0,t]$ and $\delta >0$ such that $c(s) \le -\delta$ for $t_0 \le s \le t_1$.   If $\beta > \zeta_2^2 - 1$ then the right side of 
 \begin{equation} \label{psi beta}
    \frac{d\tilde{\psi}(s)}{ds} = -1 - \beta \cos^2 \tilde{\psi}(s)- \zeta_2 \sin 2\tilde{\psi}(s), \quad t_0 \le s \le t_1
    \end{equation}
is never zero.  Separating variables and integrating gives
   $$
   \arctan\left(\frac{\tan \tilde{\psi}(s) +\zeta_2}{\sqrt{\beta + 1-\zeta_2^2}}\right) = -\bigl(\sqrt{\beta + 1-\zeta_2^2}\bigr)s + \mbox{ constant}.
   $$
A transition from $\tilde{\psi} = \pi/2$ to $\tilde{\psi} = -\pi/2$ (or vice-versa) changes the left side by $\pi$, and so the time taken for $\tilde{\psi}$ to move distance $\pi$ is $\pi/\sqrt{\beta + 1-\zeta_2^2}$.  Therefore $\tilde{\psi}(t_1) - \tilde{\psi}(t_0) \to -\infty$ as $\beta \to \infty$, and then by comparison $\psi(t_1) - \psi(t_0) \to -\infty$ as $\alpha \to \infty$.
  
 Since the right side of \eqref{psi c} is negative whenever $\psi(s) \equiv \pi/2$ mod $\pi$, then on any subinterval $[t_2,t_3]$ we have $\psi(t_3)-\psi(t_2) \le \pi$.  Therefore
    \begin{align*}
     \psi(t) & = \psi_0 + \Bigl(\psi(t_0)-\psi(0)\Bigr) + \Bigl(\psi(t_1) - \psi(t_0) \Bigr) + \Bigl(\psi(t) - \psi(t_1) \Bigr)\\
     & \le \psi_0 + 2\pi +  \Bigl(\psi(t_1) - \psi(t_0) \Bigr) \to -\infty
    \end{align*}
as $\alpha \to \infty$.  Since $\psi(t)$ depends continuously on $\alpha$, it follows by the intermediate value theorem that $\psi(t)$ takes values of the form $2\pi n + \psi_1$ for infinitely many $\alpha$ as $\alpha \to \infty$, and we are done.  
\end{proof}

\begin{proposition} \label{prop cont} Assume $a \in\R^d$ and $\gamma \in \R^m$ are not both 0.  Assume also that $(A,B)$ is a controllable pair.  Given $(v_0,\psi_0)$ and $(v_1,\psi_1)$ in $M$ and $t > 0$ there exists a piecewise continuous function $b:[0,t] \to \R^m$ such that the path $p:[0,t] \to M$ defined by 
    \begin{equation} \label{cont p}
    \frac{dp(s)}{ds} = V_0(p(s))+ \sum_{\ell = 1}^m b_\ell(s)V_\ell(p(s)), \quad 0 \le s \le t
    \end{equation}
with $p(0) = (v_0,\psi_0)$ has $p(t) = (v_1,\psi_1)$.
\end{proposition}

\begin{proof}
 The controllable pair condition allows us to control the $v$ component of a path.  There is a control path $p_1$, say, which sends $v_0$ to $0$ in the time interval $[0,t/3]$. It sends $(v_0,\psi_0)$ to $(0,\psi_2)$ for some $\psi_2 \in \R/(2 \pi \Z)$.  There is also a control path $p_2$, say, which sends 0 to $v_1$ in the time interval $[2t/3,t]$.    It induces an invertible linear mapping of $\R^2$, and so there is $\psi_3 \in \R/(2 \pi \Z)$  so that $p_2$ sends $(0,\psi_3)$ to $(v_1,\psi_1)$.  It only remains to find a path during the middle interval $[t/3,2t/3]$ which sends $(0,\psi_2)$ to $(0,\psi_3)$.  Equivalently, we see that it is sufficient to prove Proposition \ref{prop cont} in the special case $v_0=v_1 = 0$.

Assume first $\gamma = 0$ and $a \neq 0$.  The controllable pair condition implies there exists $1 \le \ell \le m$ and $0 < t_0 < t$ such that $\langle a,e^{sA}Be_\ell \rangle \neq 0$ for $0 < s  \le t_0$.  Choose control function $s \mapsto b(s)e_\ell$ for a piecewise continuous $b:[0,t] \to \R$.  The $v$ component of \eqref{cont p} is 
    $$
    \frac{dv(s)}{ds} = Av(s) + b(s)Be_\ell, \quad v(0) = 0
    $$
and this has solution
   $$
   v(t)  = \int_0^t b(s)\Bigl(e^{(t-s)A} Be_\ell\Bigr)ds.
   $$
In order to satisfy $v(t) = 0$ we require 
  \begin{equation} \label{b cond}
  \int_0^t b(s)\Bigl(e^{(t-s)A} Be_\ell\Bigr)_i ds = 0, \quad i = 1,2,\ldots,d.
  \end{equation}
The $\psi$ component of \eqref{cont p} is 
   $$
   \frac{d \psi(s)}{ds} = -1+(1-\kappa^2 + \langle a,v(s)\rangle) \cos^2 \psi(s) -  \zeta_2 \sin 2\psi(s).
   $$   
Since 
  $$
  \langle a,v(s) \rangle = \int_0^s b(r) \langle a,e^{(s-r)A}Be_\ell \rangle dr
  $$
and $\langle a,e^{sA}Be_\ell \rangle \neq 0$ for $0 < s \le t_0$, there exists a piecewise continuous $b:[0,t_0] \to \R$ 
such that 
$\langle a,v(s)\rangle \neq 0$.  Now extend $b$ to the full interval $[0,t]$ in such a way that  \eqref{b cond} is satisfied, and apply Lemma \ref{lem cont2}.  It suffices to replace the function $b$ by $\alpha b$ for suitably chosen $\alpha$.  This completes the proof in the case that $\gamma = 0$ and $a \neq 0$.

Now suppose instead that $\gamma \neq 0$.  There exists a linear mapping $T: \R^m \to \R^d$ such that $T\gamma = a$.  Define functions $f_\ell(v) = \langle Te_\ell,v\rangle$. Then
   $$
   \sum_{\ell = 1}^m f_\ell (v) V_\ell(v,\psi) = \begin{pmatrix} BT^\ast v \\[1ex] \langle a,v \rangle \cos^2 \psi \end{pmatrix}.
   $$ 
and so 
   $$
   V_0(v,\psi) - \sum_{\ell = 1}^m f_\ell (v) V_\ell(v,\psi) = \begin{pmatrix} (A-BT^\ast) v \\[1ex] -1+(1-\kappa^2) \cos^2 \psi  - \zeta_2 \sin 2 \psi  \end{pmatrix}.
   $$    
Fix  $1 \le k \le m$ such that $\gamma_k \neq 0$.  Consider a piecewise continuous function $b:[0,t] \to \R$ to be chosen later, and consider the controlled path along the time dependent vector field
     $$
      V_0(v,\psi) - \sum_{\ell = 1}^m f_\ell (v) V_\ell(v,\psi)+ b(s)V_k(v,\psi).
      $$   
The $v$ component of \eqref{cont p} is now 
     $$
    \frac{dv(s)}{ds} = (A-BT^\ast)v(s) + b(s)Be_k ,\quad v(0) = 0
    $$
and this has solution
   $$
   v(t)  = \int_0^t b(s) \Bigl( e^{(t-s)(A-BT^\ast)}Be_k\Bigr)ds.
   $$
In order to satisfy $v(t) = 0$ we require 
  \begin{equation} \label{b cond2}
  \int_0^t b(s)\Bigl(e^{(t-s)(A-BT^\ast)} Be_k\Bigr)_i ds = 0, \quad i = 1,2,\ldots,d.
  \end{equation}
The $\psi$ component of \eqref{cont p} is now
   $$
   \frac{d \psi(s)}{ds} = - 1+(1-\kappa^2 +b(s) \gamma_k)\cos^2 \psi(s) - \zeta_2 \sin 2 \psi(s).
   $$   
In order to apply Lemma \ref{lem cont2} it suffices to choose a non-identically zero function $b$ satisfying \eqref{b cond2}, and the proof is completed as before. 
\end{proof}

\subsection{Transition probabilities and invariant measures}   We use the results of the previous two sections to obtain results about the diffusion process $\{(v(t),\psi(t)): t \ge 0\}$ on $M = \R^d \times (\R/(2 \pi \Z)$.

\s

\n{\it Proof of Proposition \ref{prop ergodic}}.
For ease of notation write $(v,\psi) = x$, and let $P_t(x,U) = \PP^x(x(t) \in U)$ denote the transition probability for the diffusion.     
The Lie algebra result in Proposition \ref{prop liealg} implies condition (E) of Ichihara and Kunita \cite{IK}.  By \cite[Theorem 3]{IK} there exists smooth $p:(0,\infty) \times M \times M \to [0,\infty)$ such that $P_t(x,U) = \int_U p(t,x,y)dy$.    
The stability (eigenvalue condition) for $A$ implies the existence of the stationary probability $\mu$ for $\{v(t): t \ge 0\}$, and hence the existence of at least one stationary probability $m$ for ${\cal A}$.  Then \cite[Theorem 4]{IK} implies $m$ has a smooth density $\rho(x)$. 

Fix an open set $U \subset M$, and let $t > 0$. The stationarity of $m$ implies
      $$
     m(U)  =  \int_M  P_t(x,U)\rho(x)dx.
     $$    
The support theorem of Stroock and Varadhan \cite{SVsupp} together with Proposition \ref{prop cont} implies $P_t(x,U) >0$ for all $x \in M$, and hence $m(U) >0$.  This implies that $m$ is unique and that $\mbox{supp}(m) = M$.  It follows (see \cite[Prop 5.1]{IK}) that $P_t(x,\cdot)$ is absolutely continuous with respect to $m$ for all $t >0$ and all $x \in M$.  

Birkhoff's ergodic theorem implies 
            $$
       \PP^x\left( \lim_{t \to \infty} \frac{1}{t} \int_0^t F(x(s))ds = \int_M Fdm \right) = 1
       $$
 for $m$-almost all $x \in M$.  Finally for fixed $x \in M$, by conditioning on behavior at time 1 
 we get
     \begin{align*}
     & \PP^x\left( \lim_{t \to \infty} \frac{1}{t} \int_0^t F(x(s))ds = \int_M Fdm \right) \\
    & \hspace{10ex} = \int \PP^y\left( \lim_{t \to \infty} \frac{1}{t} \int_0^t F(x(s))ds = \int_M Fdm \right)P_1(x,dy)\\
     & \hspace{10ex} = 1              
\end{align*}
because $P_1(x,\cdot)$ is absolutely continuous with respect to $m$ and 
    $$
       \PP^y\left( \lim_{t \to \infty} \frac{1}{t} \int_0^t F(x(s))ds = \int_M Fdm \right) = 1
       $$
 for $m$-almost all $y \in M$.    \qed

\section{Proofs for Section \ref{sec scaling}} \label{sec lamproof}

In the proofs we will assume that $a \in \R^d$ and $\gamma \in \R^m$ are not both zero.  The case when $a$ and $\gamma$ are both zero is equivalent to setting $\eps = 0$.  Then \eqref{phi xi eps} has non-random constant coefficients and an elementary eigenvalue calculation gives $\lambda(\eps) \equiv  -\zeta_2$.

\subsection{Khas'minskii's formula revisited}  \label{sec Khas rev}

Recall $\kappa_d = \sqrt{\kappa^2 - \zeta_2^2}$ denotes the damped frequency of the pendulum. 
 For the asymptotic analysis as $\eps \to 0$ it is convenient to replace $u(t) = \begin{pmatrix} \varphi(t) \\ \dot{\varphi}(t) \end{pmatrix} $ with $\tilde{u}(t) = \begin{pmatrix} \kappa_d \varphi(t) \\ \zeta_2 \varphi(t) + \dot{\varphi}(t) \end{pmatrix}$, so that $\varphi(t) = \tilde{u}_1(t)/\kappa_d$ and $\dot{\varphi}(t) = \tilde{u}_2(t) - (\zeta_2/\kappa_d \tilde{u}_1(t)$.  Then we get the 2-dimensional linear SDE
     \begin{equation}  \label{u2 xi}
     d\tilde{u}(t)
        = \begin{pmatrix} -\zeta_2 & \kappa_d \\[1ex] -\kappa_d + \dfrac{\eps}{\kappa_d} \langle a,v(t)\rangle  & - \zeta_2 \end{pmatrix} \tilde{u}(t)  dt  +\frac{\eps}{\kappa_d}\sum_{\ell = 1}^m  \begin{pmatrix} 0 & 0 \\[1ex] \gamma_\ell & 0\end{pmatrix}\tilde{u}(t) dW_\ell(t).
               \end{equation} 
Write $\tilde{u}(t) = \|\tilde{u}(t)\|\begin{pmatrix} \cos \tilde{\psi}(t) \\ \sin \tilde{\psi}(t) \end{pmatrix}$.  The transformation $\psi(t) \mapsto \tilde{\psi}(t)$ is given by a diffeomorphism of $\R/(2 \pi \Z)$, so the ergodicity result Proposition \ref{prop ergodic} applies equally well to the diffusion $\{(v(t),\tilde{\psi}(t)): t \ge 0\}$.  Also $\bigl|\log \|u(t)\| - \log \|\tilde{u}(t)\|\bigr|$ is bounded.   Thus the method used to obtain formula \eqref{lam Khas} for $\lambda$ in Corollary \ref{cor ergodic} is equally valid when applied to the process $\{\tilde{u}(t): t \ge 0\}$.     

   Applying It\^{o}'s formula to \eqref{u2 xi} gives
  $$
  d \log\|\tilde{u}(t)\| =  \widetilde{Q}^\eps(v(t),\tilde{\psi}(t))dt + \sum_{\ell = 1}^m\frac{\eps \gamma_\ell}{\kappa_d} \sin \tilde{\psi}(t) \cos \tilde{\psi}(t)\, dW_\ell(t)
  $$
and 
  \begin{equation} \label{tilde psi}
   d \tilde{\psi}(t) = \tilde{h}^\eps(v(t),\tilde{\psi}(t))dt +\sum_{\ell = 1}^m \frac{\eps \gamma_\ell}{\kappa_d} \cos^2 \tilde{\psi}(t) \,dW_\ell(t)
   \end{equation}
where
  $$
  \widetilde{Q}^\eps(v,\psi) =  -\zeta_2 +\frac{\eps}{\kappa_d}\langle a,v\rangle \sin\psi \cos \psi  +\frac{\eps^2\|\gamma\|^2}{2\kappa_d^2} \cos^2 \psi \cos 2\psi
  $$
and
   $$
   \tilde{h}^\eps(v,\psi) = -\kappa_d +\frac{\eps}{\kappa_d}\langle a,v\rangle \cos^2 \psi  - \frac{\eps^2 \|\gamma\|^2}{\kappa_d^2} \sin \psi \cos^3 \psi.
   $$ 
 Repeating the arguments in Section \ref{sec Khas} leading up to Corollary \ref{cor ergodic} we get
     \begin{equation} \label{Khas eps}
   \lambda(\eps) = \int_M \widetilde{Q}^\eps(v,\psi) d\tilde{m}^\eps(v,\psi).
   \end{equation}
where $\tilde{m}^\eps$ is the unique invariant probability measure for the diffusion $\{(v(t),\tilde{\psi}(t)): t \in \R\}$ on $M$ with generator $\widetilde{\cal A}^\eps$ given by (\ref{v},\ref{tilde psi}).  For convenience of notation we drop tildes for the rest of this section.

\s

\n{\it Proof of Proposition \ref{prop lam upper}}.   For this proof we take $\eps = 1$.  We have 
   \begin{align*}
    Q(v,\psi) & =  -\zeta_2 +\frac{1}{\kappa_d}\langle a,v\rangle \sin\psi \cos \psi  +\frac{\|\gamma\|^2}{2\kappa_d^2} \cos^2 \psi \cos 2\psi \\
     & \le -\zeta_2 +\frac{1}{2\kappa_d}|\langle a,v\rangle|  +\frac{\|\gamma\|^2}{2\kappa_d^2}.
     \end{align*} 
and so
  \begin{align*}
  \lambda = \int_M Q(v,\psi) dm(v,\psi) & \le -\zeta_2 + \frac{1}{2 \kappa_d}\int_{\R^d} |\langle a,v \rangle|d\mu(v) + \frac{\|\gamma\|^2}{2\kappa_d^2}\\
  & \le -\zeta_2 + \frac{1}{2 \kappa_d}\left(\int_{\R^d} \langle a,v \rangle^2 d\mu(v)\right)^{1/2} + \frac{\|\gamma\|^2}{2\kappa_d^2}\\
  & =  -\zeta_2 + \frac{\sqrt{\langle a,Ra \rangle}}{2 \kappa_d} + \frac{\|\gamma\|^2}{2\kappa_d^2}.
  \end{align*} 
For the block and pendulum we have $R = \dfrac{\nu^2}{4 \zeta_1 }\begin{pmatrix} 1/\chi^2 & 0 \\ 0 & 1 \end{pmatrix}$ and $a = \begin{pmatrix} -\chi^2 \\-2\zeta_1 \end{pmatrix}$ and $\gamma = \nu$. \qed

\subsection{Adjoint expansion}  \label{sec adj}    
  Instead of evaluating the right side of \eqref{Khas eps} directly, we will consider the adjoint equation ${\cal A}^\eps F^\eps = Q^\eps-\lambda(\eps)$ using the asymptotic expansion method originated by Arnold, Papanicolaou and Wihstutz \cite{APW86}.    
The integrand $Q^\eps$ in \eqref{Khas eps} can be written $Q^\eps = -\zeta_2+ \eps Q_1+\eps^2 Q_2$ where   
    $$
    Q_1(v,\psi) = \frac{1}{\kappa_d}\langle a,v\rangle \sin \psi \cos \psi \qquad \mbox{ and  } \qquad
       Q_2(v,\psi) = \frac{\|\gamma\|^2}{2 \kappa_d^2}\cos ^2 \psi \cos 2 \psi.
    $$
The diffusion process $\{v(t): t \in \R\}$ given by \eqref{v} has generator ${\cal G}$ given by
   \begin{equation} \label{G}
    {\cal G}f(v) = \sum_{j=1}^d (Av)_j \frac{\partial f}{\partial v_j}(v) + \frac{1}{2} \sum_{j,k=1}^d (BB^\ast)_{jk} \frac{\partial^2 f}{\partial v_j \partial v_k}.
    \end{equation}  
 For $j = 1,\ldots,d$ we have 
 \begin{equation}  \label{dvdpsi}
   dv_j(t)d\psi(t)  =  \sum_{\ell = 1}^m (Be_\ell)_j \left(\frac{\eps \gamma_\ell}{\kappa_d} \cos^2 \psi(t)\right)
   = \frac{\eps}{\kappa_d}\cos^2\psi(t)(B\gamma)_j.
  \end{equation}
Therefore the generator ${\cal A}^\eps$ of the diffusion $\{(v(t),\psi(t)): t \in \R \}$ can be written ${\cal A}^\eps = {\cal A}_0 + \eps {\cal A}_1+\eps^2{\cal A}_2 $ where 
      $$
      {\cal A}_0 = {\cal G}-\kappa_d \frac{\partial}{\partial \psi}  
      $$
 and 
  $$
  {\cal A}_1 = \frac{1}{\kappa_d} \cos^2 \psi \left(\langle a,v\rangle\frac{\partial}{\partial \psi}+  \sum_{j=1}^d (B\gamma)_j \frac{\partial^2}{\partial \psi \partial v_j}\right)
  $$
and 
   $$
  {\cal A}_2 = \frac{\|\gamma\|^2}{\kappa_d^2}\left(-\sin \psi \cos^3 \psi  \frac{\partial}{\partial \psi} +  \cos^4 \psi \frac{\partial^2}{\partial \psi^2}\right).
  $$      
Expanding the adjoint equation ${\cal A}^\eps F^\eps = Q^\eps-\lambda(\eps)$
as
    $$    
    ({\cal A}_0+\eps {\cal A}_1+\eps^2 {\cal A}_2)(\eps F_1+\eps^2F_2+\cdots) = (-\zeta_2+\eps Q_1+\eps^2 Q_2) - (\lambda_0+ \eps \lambda_1 + \eps^2\lambda_2+\cdots)
    $$ 
and equating powers of $\eps$ we get $\lambda_0 = -\zeta_2$ and 
    \begin{align}
       {\cal A}_0 F_1 & = Q_1-\lambda_1  \label{adj1}\\
       {\cal A}_0F_2 +{\cal A}_1 F_1 & = Q_2-\lambda_2 \label{adj2}\\
       {\cal A}_0F_3 +{\cal A}_1 F_2 + {\cal A}_2F_1 & = -\lambda_3 \label{adj3},
  \end{align}
and so on.  In the next three sections we will solve these equations, finding explicit values for $\lambda_1$, $\lambda_2$ and $\lambda_3$.  We will find an explicit formula for $F_1$ and characterizations of the functions $F_2$ and $F_3$ which are sufficiently precise so as to enable a rigorous asymptotic estimate to be made in Section \ref{sec proof}.

\begin{remark}  Solving the system recursively, all the equations will be of the form 
      $$
      {\cal A}_0 F(v,\psi) = G(v,\psi) - \lambda
      $$
where $G$ is given and $\lambda$ and $F$ are to be found.  For an exponentially ergodic Markov process with generator ${\cal L}$, semigroup $\{P_t: t \ge 0\}$, and invariant probability $m$, then $P_tG \to \int Gdm$
exponentially fast as $t \to \infty$.  Then taking $\lambda = \int Gdm$ and 
    \begin{equation} \label{badF}
    F = -\int_0^\infty (P_tG-\lambda)dt
    \end{equation}
solves ${\cal L}F = G-\lambda$.  But this is not the case for ${\cal A}_0$; for example $G(v,\psi) = \cos \psi$ has $\lambda = \int Gdm = 0$ but $P_t G(v,\psi) = \cos (\psi - \kappa_d t)$ so that the integral for $F$ does not converge.  This causes some extra work, especially when solving the second and third equations.  The construction of $D_n$ in Lemmas \ref{lem CtoDE} and \ref{lem CtoDE3} is related to, but distinctly different from, the construction in \eqref{badF}. 
\end{remark}

\subsubsection{The first equation} \label{sec first}
We need to find $\lambda_1$ and $F_1(v,\psi)$ so that  
$$ {\cal A}_0F_1(v,\psi) = Q_1(v,\psi) - \lambda_1
  $$
where 
 $$Q_1(v,\psi) =  \frac{1}{2\kappa_d}\langle a,v\rangle \sin 2 \psi  = -\frac{1}{2 \kappa_d}\Re\Bigl( i \langle a,v\rangle e^{i 2\psi}\Bigr).
   $$
Define 
   \begin{equation} \label{F1}
    F_1(v,\psi) = \frac{1}{2 \kappa_d}\Re\biggl( ie^{i2\psi} \int_0^\infty \langle a,e^{tA}v\rangle e^{-i 2 \kappa_d t}dt \biggr)
   \end{equation} 
Here $\Re$ denotes the real part of the complex valued expression in parentheses.  Since the eigenvalues of $A$ have negative real parts, the integrand decays exponentially quickly with $t$ so that $F_1(v,\psi)$ is well-defined.  Moreover it is linear in $v$ and we may differentiate with respect to $v$ inside the integral.  We get 
    \begin{align}
    {\cal A}_0 F_1(v,\psi) & =  \frac{1}{2 \kappa_d}\Re\biggl( ie^{i2\psi}\int_0^\infty \Bigl(\langle a,e^{tA}Av\rangle - i2 \kappa_d \langle a,e^{tA}v\rangle \Bigr) e^{-i 2\kappa_d t}dt \biggr) \nonumber \\
    & =  \frac{1}{2 \kappa_d}\Re\biggl( ie^{i2\psi}\int_0^\infty \frac{\partial}{\partial t}\Bigl(\langle a,e^{tA}v\rangle e^{-i 2\kappa_d t}\Bigr)dt \biggr) \nonumber \\
    & =  - \frac{1}{2 \kappa_d}\Re\Bigl( ie^{i2\psi} \langle a,v\rangle \Bigr) \nonumber \\
    & = Q_1(v,\psi). \label{F1Q1}
     \end{align} 
Therefore $\lambda_1 = 0$.
    For ease of notation we write 
 \begin{equation}  \label{b}
    b = \biggl( \int_0^\infty e^{tA^\ast}e^{-i2\kappa_d t} dt \biggr)a \in \mathbb{C}^d.
    \end{equation}   
so that
 \begin{equation}
     F_1(v,\psi) =  \frac{1}{2 \kappa_d}\Re\biggl( ie^{i2\psi}\int_0^\infty \langle e^{tA^\ast}a,v\rangle e^{-i 2 \kappa_d t}dt \biggr) =  \frac{1}{2 \kappa_d} \Re\Bigl( ie^{i2\psi} \langle b, v\rangle \Bigr). \label{F1b}
   \end{equation}

\subsubsection{The second equation} \label{sec second}
We need to find $\lambda_2$ and $F_2(v,\psi)$ such that
      $$
     {\cal A}_0F_2(v,\psi) = -{\cal A}_1F_1(v,\psi) + Q_2(v,\psi) - \lambda_2.
     $$
   Using \eqref{F1b} we have 
    \begin{align*}
    {\cal A}_1F_1(v,\psi) & =  \frac{1}{\kappa_d }\cos^2 \psi \Bigl(\langle a,v \rangle + \sum_{j=1}^d (B\gamma)_j \frac{\partial}{\partial v_j}\Bigr) \frac{\partial F_1}{\partial \psi}\\
    & =  -\frac{1}{\kappa_d ^2}\cos^2 \psi \,\Re\Bigl(  e^{i2 \psi}\bigl(\langle a,v\rangle \langle b,v\rangle  + \langle b,B\gamma \rangle \bigr) \Bigr)\\
     & =  -\frac{1}{4\kappa_d ^2}\Re\Bigl(  \bigl(1+2e^{i 2\psi} + e^{i4\psi} \bigr)\bigl(\langle a,v\rangle \langle b,v\rangle  + \langle b,B\gamma \rangle \bigr) \Bigr).
    \end{align*}
Also
 $$
    Q_2(v,\psi)   =   \frac{\|\gamma\|^2}{2 \kappa_d^2}\cos ^2 \psi \cos 2 \psi =  \frac{\|\gamma\|^2}{8 \kappa_d^2}\Re \bigl(1+2e^{i2\psi} + e^{i 4\psi}\bigr).
   $$
So we need to solve
   \begin{equation}
   {\cal A}_0F_2(v,\psi)   =  \frac{1}{4\kappa_d^2}\Re\Bigl(\bigl(1+2e^{i 2\psi} + e^{i4\psi} \bigr)\bigl(\langle a,v\rangle \langle b,v\rangle  + \langle b,B\gamma \rangle  + \frac{\|\gamma\|^2}{2}\bigr)\Bigr)  - \lambda_2. \label{adj2bis}
   \end{equation}    

Let ${\cal F}_2$ denote the space of functions $M \to \mathbb{C}$ consisting of complex linear combinations of functions of the form $e^{i n \psi}C(v,v)$ and $e^{i n \psi}$ where $C$ is bilinear and $n= 0,2,4$.  The formulas above show that $-{\cal A}_1F_1(v,\psi) + Q_2(v,\psi) \in {\cal F}_2$.  Then we can break down the problem of solving \eqref{adj2bis} into a collection of (complex-valued) problems of the form 
      \begin{equation} \label{AFCn}
       {\cal A}_0 F(v,\psi) =e^{i n \psi} C(v,v) - \lambda
      \end{equation}
and  
      \begin{equation} \label{AFn}
      {\cal A}_0F(v,\psi) = e^{in\psi} - \lambda
      \end{equation}   
The problem \eqref{AFn} is trivial: if $n \neq 0$ take $F(v,\psi) = (i/n \kappa_d)e^{i n \psi}$ and $\lambda = 0$, and if $n= 0$ take $F(v,\psi) = 0$ and $\lambda = 1$.  The problem \eqref{AFCn} is more interesting and the following lemma, generalizing the calculation in \eqref{F1Q1} to the bilinear setting, will be useful.
 
\begin{lemma} \label{lem CtoDE}  Suppose $C: \R^d \times \R^d \to \mathbb{C}$ is bilinear.  For $n \ge 0$ define
    \begin{equation} \label{Dn}
    D_n(v^{(1)},v^{(2)}) = - \int_0^\infty C( e^{tA}v^{(1)},e^{tA}v^{(2)})e^{-i n \kappa_d t}dt.
   \end{equation}
Then $D_n$ is well defined and bilinear and there exists $E_n \in \mathbb{C}$ such that 
     $$
     {\cal A}_0\bigl(e^{i n \psi }D_n(v,v)\bigr) = e^{in \psi} C(v,v) + e^{i n  \psi} E_n.
     $$
Moreover
 \begin{equation} \label{E0}
 E_0 = -\sum_{\ell=1}^m \int_0^\infty C(e^{tA}Be_\ell, e^{tA}Be_\ell)dt  = -\sum_{j,k=1}^d C(e_j,e_k)R_{jk}= -\int_{\R^d} C(v,v)d\mu(v)
\end{equation} 
 where $R$ is the variance matrix for the invariant probability measure $\mu$. 
 \end{lemma}
  

\begin{proof}  $D_n$ is well defined because the eigenvalues of $A$ have negative real parts.  We will break down the calculation of ${\cal A}_0\bigl(e^{i n \psi }D_n(v,v)\bigr)$ into first and second order derivatives.   
We have 
   \begin{align*}
  \lefteqn{ \Bigl(\sum_{j=1}^d (Av)_j \frac{\partial}{\partial v_j} - \kappa_d \frac{\partial}{\partial \psi}\Bigr)\bigl(e^{i n \psi }D_n(v,v)\bigr)} \hspace{5ex}
   \\ 
   &  =   e^{in  \psi}\Bigl(D_n(Av,v)+D_n(v,Av) - in\kappa_d  D_n(v,v)\Bigr)\\
    &   =
   - e^{in\psi}\int_0^\infty \Bigl(C(Ae^{tA}v,e^{tA}v)+C(e^{tA}v,Ae^{tA}v) - in \kappa_d  C(e^{tA}v,e^{tA}v)\Bigr)e^{-in \kappa_d  t}dt   \\
    &  =
   - e^{in\psi}\int_0^\infty \frac{\partial}{\partial t}\Bigl(C(e^{tA}v,e^{tA}v)e^{-in \kappa_d t}\Bigr)dt\\
   &   = e^{in\psi}C(v,v),
    \end{align*}
and
   $$
    \frac{1}{2} \sum_{j,k = 1}^d (BB^\ast)_{jk} \frac{\partial^2 D_n}{\partial v_j \partial v_k}(v,v) 
     = \frac{1}{2} \sum_{j,k = 1}^d (BB^\ast)_{jk}\Bigl(D_n(e_j,e_k)+D_n(e_k,e_j)\Bigr):=E_n
       $$ 
where $e_j$ and $e_k$ denote the $j$th and $k$th standard unit vectors in $\R^d$.  
Together we have 
  \begin{align*}
     {\cal A}_0\bigl(e^{i n \psi }D_n(v,v)\bigr)
        & =  \Bigl(\sum_{j=1}^d (Av)_j \frac{\partial}{\partial v_j} - \kappa_d \frac{\partial}{\partial \psi}\Bigr)\bigl(e^{i n \psi }D_n(v,v)\bigr) +  \frac{e^{i n \psi}}{2} \sum_{j,k = 1}^d (BB^\ast)_{jk} \frac{\partial^2 D_n}{\partial v_j \partial v_k}(v,v) \\
        & =  e^{in  \psi}C(v,v) + e^{in \psi }E_n.
 \end{align*}
Finally
  \begin{align*}
  E_0 & =  -\frac{1}{2} \sum_{j,k = 1}^d (BB^\ast)_{jk}\int_0^\infty \Bigl(C(e^{tA}e_j,e^{tA}e_k)+C(e^{tA}e_k,e^{tA}e_j)\Bigr)dt\\
  & =   -\sum_{\ell = 1}^m \int_0^\infty C(e^{tA}Be_\ell,e^{tA}Be_\ell)dt\\
   & =  -\sum_{j,k=1}^d C(e_j,e_k)\sum_{\ell = 1}^m \int_0^\infty \bigl(e^{tA}Be_\ell\bigr)_j\bigl(e^{tA}Be_\ell\bigr)_k dt,
  \end{align*}
and the rest of \eqref{E0} follows immediately from \eqref{R}. 
\end{proof}

\s

The lemma implies that at the cost of translating the function $F$ by an element in ${\cal F}_2$ we can reduce a problem of the form \eqref{AFCn} to a problem of the form \eqref{AFn}.  For $n \neq 0$ we can solve \eqref{AFCn} with $F \in {\cal F}_2$ and $\lambda = 0$.  The remaining case of interest is \eqref{AFCn} with $n= 0$, which has a solution with $F \in {\cal F}_2$ and $\lambda = -E_0$.   Since the right side of \eqref{adj2bis} contains a bilinear term with $C(v,v) = \langle a,v\rangle \langle b,v\rangle$ we note that the corresponding value of $-E_0$ from \eqref{E0} is 
    $$
   -E_0 = \sum_{j,k=1}^d \langle a,e_j\rangle \langle b,e_j\rangle R_{jk} = \sum_{j,k=1}^d a_jb_kR_{jk} = \langle b,Ra \rangle.
    $$ 
Putting all the calculations together, and going through the right side of \eqref{adj2bis} term by term, we obtain the following result.
\begin{proposition}  \label{prop adj2}The equation \eqref{adj2bis} has a solution with $F_2 = \Re(\widehat{F}_2)$ for some $\widehat{F}_2 \in {\cal F}_2$ and 
    \begin{equation} \label{lam2b}
    \lambda_2 =\frac{1}{4\kappa_d^2}\Bigl( \langle \Re( b),Ra \rangle + \langle \Re(b),B\gamma \rangle + \frac{\|\gamma\|^2}{2} \Bigr).
    \end{equation}
\end{proposition}

Finally, using \eqref{b} and \eqref{hatS} we have 
   $$\langle \Re(b),Ra\rangle = \Re\left\langle a, \int_0^\infty e^{tA} e^{-i 2 \kappa_d t}Ra \right\rangle
     = \pi \langle a, \widehat{S}_A(2 \kappa_d)Ra \rangle 
   $$ 
and similarly
  $$\langle \Re(b),B\gamma \rangle 
     = \pi \langle a, \widehat{S}_A(2\kappa_d )B\gamma \rangle, 
   $$ 
so that 
  \begin{equation} \label{lam2 bis}
   \lambda_2 = \frac{\pi}{4 \kappa_d^2} \Bigl( \langle a,\widehat{S}_A(2\kappa_d )Ra\rangle  + \langle a,\widehat{S}_A(2\kappa_d) B\gamma \rangle +\frac{\|\gamma\|^2}{2\pi}\Bigr).
  \end{equation}
This is the formula for $\lambda_2(2 \kappa_d)$ in Theorem \ref{thm lam}, but the asymptotic behavior of $\lambda(\eps)$ has not yet been proved.  

 \subsubsection{The third equation} \label{sec third}
Next we consider 
  \begin{equation}  \label{adj3bis}
 {\cal A}_0F_3 = -{\cal A}_1 F_2 - {\cal A}_2F_1   -\lambda_3
  \end{equation}
Let ${\cal F}_3$ denote the space of functions $M \to \mathbb{C}$ consisting of complex linear combinations of functions of the form $e^{i n \psi}C(v,v,v)$ and $e^{i n \psi}D(v)$ where $C$ is trilinear and $D$ is linear and $n= 0,2,4,6$.  The exact formula \eqref{F1b} for $F_1$ and the characterization of $F_2$ in Proposition \ref{prop adj2} together imply that $ -{\cal A}_1 F_2 - {\cal A}_2F_1$ is the real part of a function in ${\cal F}_3$. 

\begin{lemma} \label{lem CtoDE3}  Suppose $C: (\R^d)^3 \to \mathbb{C}$ is trilinear.  For $n \ge 0$ define
    $$
    D_n(v^{(1)},v^{(2)},v^{(3)}) = - \int_0^\infty C( e^{tA}v^{(1)},e^{tA}v^{(2)},e^{tA}v^{(3)})e^{-i n \kappa_d t}dt.
    $$
Then $D_n$ is well defined and trilinear and there exists a linear mapping $E_n: \R^d \to \mathbb{C}$ such that 
     $$
     {\cal A}_0\bigl(e^{i n \psi }D_n(v,v,v)\bigr) = e^{in \psi} C(v,v,v) + e^{i n  \psi} E_n(v).
     $$
 \end{lemma}

\begin{proof}  The method of proof is the same as for Lemma \ref{lem CtoDE}, and we omit the details.  
\end{proof}

The lemma implies that at the cost of translating the function $F$ by an element of ${\cal F}_3$ we can reduce a problem of the form
     \begin{equation} \label{AFCn3}
       {\cal A}_0 F(v,\psi) =e^{i n \psi} C(v,v,v)
      \end{equation}
for $n= 0,2,4,6$ into a problem of the form 
      \begin{equation} \label{AFEn3}
       {\cal A}_0 F(v,\psi) =e^{i n \psi} E_n(v).
      \end{equation}
The method used in Section \ref{sec first} to obtain an explicit solution of \eqref{adj1} can be equally applied here to obtain an explicit solution for \eqref{AFEn3} with $F \in {\cal F}_3$.   Repeating the term by term approach used in the previous section we get the next result.  

\begin{proposition}  \label{prop adj3}The equation \eqref{adj3bis} has a solution with $F_3 = \Re(\widehat{F}_3)$ for some $\widehat{F}_3 \in {\cal F}_3$ and $\lambda_3 = 0$. 
   \end{proposition} 
   
\subsection{Proof of Theorem \ref{thm lam}} \label{sec proof}

When applying the adjoint method on a non-compact space such as $M$ the following result of Baxendale and Goukasian \cite[Prop 3]{BG01} is important.

\begin{proposition} \label{prop BG}
  Let $\{X(t):\ t\geq 0\}$ be a
diffusion process on a $\sigma$-compact manifold $M$ with
invariant probability measure $m$. Let ${\cal B}$ be an operator acting on $C^2(M)$ functions that agrees with the generator of $\{X(t): t\geq 0\}$ on $C^2$ functions with compact support.  Let $f \in
C^2(M)$, and assume $f$ and ${\cal B}f$ are $m$-integrable.  Suppose there exists a positive $G \in C^2(M)$ and $k <\infty$ satisfying ${\cal B}G(x) \le kG(x)$ and ${\cal B}f(x) \leq G(x)$ for all $x \in M$, and $f(x)/G(x) \to 0$ as $x \to \infty$. Then $\int_M {\cal B}f(x)dm(x) = 0$.
 \end{proposition}

\begin{remark}  For an example as to why we need something like Proposition \ref{prop BG} consider 
     $$
     dX_t = \left(-aX_t^3+ \frac{\sigma^2}{2X_t}\right)dt + \sigma dW_t
     $$
on $M = (0,\infty)$ with $f(x) = \log x$.  We get ${\cal B}f(x) = -ax^2$ but clearly $\int_0^\infty (-ax^2)dm(x) \neq 0$.
\end{remark}

\n{\it Proof of Theorem \ref{thm lam}(i)}. 
 So far we have shown the existence of functions $F_1$, $F_2$ and $F_3$ satisfying (\ref{adj1},\ref{adj2},\ref{adj3}) with $\lambda_1 = \lambda_3 = 0$ and $\lambda_2$ given by the formula \eqref{lam2 bis}.  Since $F_1$, $F_2$ and $F_3$ are smooth functions of $(v,\psi)$ we have 
   \begin{align}
   {\cal A}^\eps(\eps F_1+\eps^2 F_2+\eps^3 F_3) & =  ({\cal A}_0+\eps{\cal A}_1+\eps^2{\cal A}_2)(\eps F_1+\eps^2 F_2 +\eps^3 F_3) \nonumber \\
      & =  \eps {\cal A}_0F_1+ \eps^2({\cal A}_0 F_2+{\cal A}_1F_1) +\eps^3({\cal A}_0 F_3+ {\cal A}_1 F_2+{\cal A}_2 F_1) \nonumber \\
      & \qquad  + \eps^4( {\cal A}_1F_3 + {\cal A}_2F_2) + \eps^5 {\cal A}_2F_3 \nonumber \\
    & =  \eps Q_1+ \eps^2(Q_2-\lambda_2)+ \eps^4( {\cal A}_1F_3 + {\cal A}_2F_2) + \eps^5 {\cal A}_2F_3  \nonumber\\
    & =  Q^\eps + \zeta_2 -\eps^2 \lambda_2+ \eps^4( {\cal A}_1F_3 + {\cal A}_2F_2) + \eps^5 {\cal A}_2F_3. \label{proof1}
    \end{align}    
Fix $\eps >0$ for the moment.  From the explicit formula \eqref{F1} for $F_1$ and the characterizations of $F_2$ and $F_3$ in Propositions \ref{prop adj2} and \ref{prop adj3} we know that $\eps F_1+\eps^2 F_2+\eps^3 F_3$ grows at most like $\|v\|^3$, and that ${\cal A}_1F_3$ grows at most like $\|v\|^4$ and that ${\cal A}_2F_2$ grows at most like $\|v\|^2$ and that ${\cal A}_2 F_3$ grows at most like $\|v\|^3$.   We also know that the $v$ marginal of $m^\eps$ is the Gaussian measure $\mu$ with mean $0$ and variance matrix $R$, so that 
  \begin{equation} \label{moment}
  \int_M \|v\|^k dm^\eps(v,\psi) = \int_{\R^d} \|v\|^k d\mu(v) < \infty
  \end{equation}
for all $k \ge 0$.  Since
   $$
   {\cal A}^\eps (\|v\|^4) = {\cal G}(\|v\|^4) = 4\langle Av,v\rangle\|v\|^2+ 2\mbox{tr}(BB^\ast)\|v\|^2 + 4\|B^\ast v\|^2
   $$
we can apply Proposition \ref{prop BG} with ${\cal B} = {\cal A}^\eps$ and $f = \eps F_1+\eps^2 F_2+\eps^3 F_3$ and $G = c_1+ c_2\|v\|^4$ for suitable positive constants $c_1$ and $c_2$.   Integrating \eqref{proof1} with respect to the invariant probability $m^\eps$ gives
   \begin{equation} \label{proof2}
    0 = \int Q^\eps dm^\eps  + \zeta_2 -\eps^2 \lambda_2  + \eps^4 \int_M ( {\cal A}_1F_3 + {\cal A}_2F_2)dm^\eps + \eps^5 \int_M {\cal A}_2F_3 dm^\eps.
    \end{equation}
Using Khas'minskii's formula \eqref{Khas eps}, this gives 
    \begin{equation} \label{proof3}
    \lambda(\eps)  = -\zeta_2 + \eps^2 \lambda_2  - \eps^4 \int_M ( {\cal A}_1F_3 + {\cal A}_2F_2)dm^\eps - \eps^5 \int_M {\cal A}_2F_3 dm^\eps.
    \end{equation}
Finally the growth estimates above on ${\cal A}_1 F_3$ and ${\cal A}_2F_2$ and ${\cal A}_2F_3$ together with \eqref{moment} imply that the integrals
     $$
     \int_M ( {\cal A}_1F_3 + {\cal A}_2F_2)dm^\eps \quad \mbox{ and } \quad \int_M {\cal A}_2F_3 dm^\eps
     $$
     are bounded uniformly in $\eps$.  At this point we let $\eps \to 0$ in \eqref{proof3} and obtain 
        $$
        \lambda(\eps) = \eps^2 \lambda_2 + O(\eps^4)
        $$
as $\eps \to 0$.  This completes the proof of Theorem \ref{thm lam}. \qed
\s

\n {\it Proof of Theorem \ref{thm lam}(ii).}  From \eqref{adj2bis} we see that ${\cal A}_0F_2$ is a finite sum of terms of the form $\Re\left(e^{i n \psi}c\right)$ or $\Re\left(e^{i n \psi}C(v,v)\right)$ where the constants $c$ and the coefficients in the bilinear mappings $C$ are continuous functions of $\kappa_d$ for $ 0  < \kappa_d < \infty$.  The construction in Section \ref{sec second} implies that the same is true for $F_2$, and then also for ${\cal A}_2F_2$.  So given $0 < c_1 < c_2 <\infty$ there exists $K_1$ such that $|{\cal A}_2F_2(\psi,v)| \le K_1(1+\|v\|^2)$ whenever $c_1 \le \kappa_d \le c_2$.  

Similarly from \eqref{adj3bis} we see that ${\cal A}_0F_3$ is a finite sum of terms of the form $\Re\left(e^{i n \psi}D(v)\right)$ or $\Re\left(e^{i n \psi}C(v,v,v)\right)$ where the coefficients in the linear mappings $D$ and trilinear mappings $C$ are continuous functions of $\kappa_d$ for $ 0  < \kappa_d < \infty$.  The construction in Section \ref{sec third} implies that the same is true of $F_2$.  So given $0 < c_1 < c_2 <\infty$ there exists $K_2$ and $K_3$ such that $|{\cal A}_1F_3(\psi,v)| \le K_2(1+\|v\|^4)$ and $|{\cal A}_2F_3(\psi,v)| \le K_3(1+\|v\|^3)$ whenever $c_1 \le \kappa_d \le c_2$.  

These estimates can now be used in \eqref{proof3} to give the uniform estimate \eqref{lam eps unif}.  \qed

\subsection{Proof of Proposition \ref{prop Vdot}} \label{sec V}

Taking $a=A^\ast \alpha$ and $\gamma = B^\ast \alpha$ in \eqref{lam2 omega} we get 
    \begin{align*}
   \lambda_2(\omega) & =  \frac{\pi}{\omega^2} \Bigl( \langle A^\ast \alpha,\widehat{S}_A(\omega)RA^\ast \alpha\rangle  + \langle A^\ast \alpha,\widehat{S}_A(\omega) BB^\ast \alpha \rangle +\frac{\|B^\ast \alpha\|^2}{2\pi}\Bigr) \\
   & =   \frac{\pi}{\omega^2} \Bigl( \langle \alpha,A\widehat{S}_A(\omega)RA^\ast \alpha\rangle  + \langle \alpha,A\widehat{S}_A(\omega) BB^\ast \alpha \rangle +\frac{\langle \alpha,BB^\ast \alpha \rangle}{2\pi}\Bigr).
  \end{align*}
The covariance matrix $R$ in \ref{R} satisfies 
  \begin{equation} \label{R2}
   AR+RA^\ast = -BB^\ast,
   \end{equation}
see Gardiner \cite[Sect 4.4]{Gard}.  Substituting for $BB^\ast$ and noting $\langle \alpha,BB^\ast \alpha \rangle = -\langle \alpha, AR\alpha \rangle - \langle \alpha,RA^\ast \alpha \rangle = -2 \langle \alpha ,AR\alpha \rangle$ we get 
      \begin{equation} \label{lambda2 V2}  
   \lambda_2(\omega) =   \frac{\pi}{\omega^2}\Bigl(-  \langle \alpha, A\widehat{S}_A(\omega)AR \alpha\rangle -\frac{\langle \alpha ,AR\alpha\rangle}{\pi} \Bigr).
   \end{equation}   
Integrating by parts twice gives         
 $$
    A \widehat{S}_A(\omega)A  =  \frac{1}{\pi}\int_0^\infty Ae^{tA}A \cos \omega t\,dt 
      =  -\frac{1}{\pi}A - \frac{\omega^2}{\pi} \int_0^\infty e^{tA}\cos \omega t\,dt = -\frac{1}{\pi}A - \omega^2 \widehat{S}_A(\omega). 
      $$
Substituting for $A\widehat{S}_A(\omega)A$ in \eqref{lambda2 V2} gives
     \begin{equation} \label{lam2 R}
     \lambda_2(\omega) = \pi \langle \alpha, \widehat{S}_A(\omega)R\alpha\rangle.
 \end{equation}
         Let $R(t)$ denote the autocovariance matrix for $\{v(t): t \in \R\}$.  Then $R(0)$ is the covariance matrix $R$ of the stationary version, and $R(t) = e^{tA}R$ for $t \ge 0$.  The autocovariance function for $V$ is
        $$
        R_V(t) = \E[V(t)V(0)] = \langle \alpha,R(t)\alpha \rangle = \langle \alpha,e^{tA}R \alpha  \rangle
        $$   
for $t \ge 0$. The power spectral density of $V$ is 
    $$
    S_V(\omega) =  \frac{1}{\pi}\int_0^\infty R_V(t) \cos \omega t \,dt  
    = \langle \alpha, \widehat{S}_A(\omega)R \alpha \rangle,
   $$
and so 
$$
     \lambda_2(\omega) = \pi \langle \alpha, \widehat{S}_A(\omega)R\alpha\rangle =  \pi S_V(\omega),
$$
as required.  \qed

\subsection{Proof of Proposition \ref{prop psd xi}}
 
We use the stationary version of $v$ given by 
    $$
    v(t) = \int_{-\infty}^t e^{(t-s)A}BdW(s) = \sum_{\ell = 1}^m \int_{-\infty}^t e^{(t-s)A} Be_\ell dW_\ell(s),
   $$
where $e_1,\ldots,e_m$ denote the standard basis vectors in $\R^m$.  Then   
   \begin{align*}
   \int_{-\infty}^\infty \langle a,v(u)\rangle \psi_\delta(t-u)du
    & =  \sum_{\ell = 1}^m \int_{-\infty}^\infty \left(\int_{-\infty}^u \left\langle a, e^{(u-s)A} Be_\ell \right\rangle dW_\ell(s)\right)\psi_\delta(t-u)du\\
    & =  \sum_{\ell = 1}^m \int_{-\infty}^\infty \left(\int_s^\infty \left\langle a,e^{(u-s)A} Be_\ell \right\rangle \,\psi_\delta(t-u)du \right)dW_\ell(s),
    \end{align*}
and so 
       $$
    \xi_\delta(t) = \sum_{\ell = 1}^m \int_{-\infty}^\infty \left( \int_s^\infty \langle a,e^{(u-s)A} Be_\ell\rangle \,\psi_\delta(t-u)du +  \gamma_\ell \psi_\delta(t-s)\right) dW_\ell(s)  \equiv \sum_{\ell = 1}^m \xi_{\delta,\ell}(t), 
    $$
say,   Since the $\{W_\ell: 1 \le \ell \le m\}$ are independent, then the $\{\xi_{\delta,\ell}: 1 \le \ell \le m\}$ are independent, and the autocovariance function for $\xi_\delta$ is the sum of the autocovariance functions for each of the $\xi_{\delta,\ell}$.  The computation of each $E[\xi_{\delta,\ell}(t)\xi_{\delta,\ell}(0)]$ can be broken down into 4 terms.  Therefore   
\begin{align*}
  \lefteqn{\E[\xi_\delta(t)\xi_\delta(0)]} \hspace{5ex}  \\
   & =  \sum_{\ell = 1}^\infty \E[\xi_{\delta,\ell}(t)\xi_{\delta,\ell}(0)]\\
   & =  \sum_{\ell = 1}^m \int_{-\infty}^\infty \left(\int_s^\infty \langle a,e^{(u-s)A} Be_\ell\rangle \,\psi_\delta(t-u)du \right)\left(\int_s^\infty \langle a,e^{(w-s)A} Be_\ell\rangle \,\psi_\delta(-w)dw \right)ds \\
   &  \qquad + \sum_{\ell = 1}^m \int_{-\infty}^\infty \left(\int_s^\infty \langle a,e^{(u-s)A} Be_\ell\rangle \,\psi_\delta(t-u)du \right) \gamma_\ell \psi_\delta(-s)ds \\
   &  \qquad + \sum_{\ell = 1}^m \int_{-\infty}^\infty \left(\int_s^\infty \langle a,e^{(u-s)A} Be_\ell\rangle \,\psi_\delta(-u)du \right) \gamma_\ell \psi_\delta(t-s)ds \\
       &  \qquad + \sum_{\ell = 1}^m \gamma_\ell^2 \int_{-\infty}^\infty \psi_\delta(t-s)\psi_\delta(-s)\,ds \\ 
     & =  \int_{-\infty}^\infty \left( \int_s^\infty \int_s^\infty  \langle a, e^{(u-s)A}BB^\ast e^{(w-s)A^\ast}a \rangle \psi_\delta(t-u)\psi_\delta(-w)dudw \right)\,ds \\
     &   \qquad +  \int_{-\infty}^\infty \left(\int_s^\infty \langle a,e^{(u-s)A} B \gamma \rangle \,\psi_\delta(t-u)du \right)  \psi_\delta(-s)ds \\ 
     &  \qquad  + \int_{-\infty}^\infty \left(\int_s^\infty \langle a,e^{(u-s)A} B\gamma\rangle \,\psi_\delta(-u)du \right)  \psi_\delta(t-s)ds \\
       & \qquad +  \int_{-\infty}^\infty \|\gamma\|^2\psi_\delta(t-s)\psi_\delta(-s)\,ds\\
       & =  \int_{-\infty}^\infty \Bigl( I_\delta^{(1)}(s,t) + I_\delta^{(2)}(s,t)+I_\delta^{(3)}(s,t) + I_\delta^{(4)}(s,t)\Bigr)ds,
   \end{align*}
say.  Then
   $$  S_{\xi_\delta}(\omega)  =  \frac{1}{\pi}\int_0^\infty \E[\xi_\delta(t) \xi_\delta(0)] \cos \omega t\,dt
   = \sum_{i=1}^4 \frac{1}{\pi} \int_0^\infty \left(\int_{-\infty}^\infty I_\delta^{(i)}(s,t)\,ds\right) \cos \omega t\,dt. $$
We consider the limit as $\delta \to 0$ for each of the four terms separately.

\paragraph{$i=1$:}
   $$
   I_\delta^{(1)}(s,t) =   \int_s^\infty \int_s^\infty  \left\langle a, e^{(u-s)A}BB^\ast e^{(w-s)A^\ast}a \right\rangle \psi_\delta(t-u)\psi_\delta(-w)\, dudw.
   $$   
Recalling $t \ge 0$ we have 
   $$
   I_\delta^{(1)}(s,t) \to \left\{ \begin{array}{cl}
        \langle a, e^{(t-s)A}BB^\ast e^{-sA^\ast}a \rangle & \mbox{ if } s < 0 \\
            0 & \mbox{ if }s > 0
                \end{array} \right.
                $$
as $\delta \to 0$.  Also $ I_\delta^{(1)}(s,t)=0$ if $s > \delta$ and 
\begin{align*}
 \left|I_\delta^{(1)}(s,t)\right| & \le \sup_{-\delta \le r_1,r_2 \le \delta} \left| \left\langle a,e^{(t+r_1-s)A}BB^\ast e^{(r_2-s)A^\ast} a \right\rangle \right|\\
 &  \le \left(\sup_{|r| \le \delta} \|e^{rA^\ast} a\|\right)^2 \|e^{(t-s)A}BB^\ast e^{-sA^\ast}\|.
\end{align*}
 Therefore by the dominated convergence theorem we have 
   $$
     \frac{1}{\pi} \int_0^\infty \left(\int_{-\infty}^\infty  I_\delta^{(1)} (s,t)\,ds\right)\cos \omega t\,dt
     \to  \frac{1}{\pi} \int_0^\infty \left( \int_{-\infty}^0  \langle a, e^{(t-s)A}BB^\ast e^{-sA^\ast}a \rangle \,ds\right) \cos \omega t\,dt
     $$
as $\delta \to 0$.  Since
$$   
 \int_{-\infty}^0 \langle a, e^{(t-s)A}BB^\ast e^{-sA^\ast}a \rangle\,ds 
  = \left\langle a, e^{tA}\left(\int_{-\infty}^0 e^{-sA}BB^\ast e^{-sA^\ast}\,ds\right)a \right\rangle
  = \langle a,e^{tA}Ra\rangle
  $$
where $R$ is the covariance matrix for the invariant probability measure for $v(t)$, see \eqref{R}, we have 
   $$
   \lim_{\delta \to 0} \frac{1}{\pi} \int_0^\infty\left( \int_{-\infty}^\infty  I_\delta^{(1)}(s,t)\,ds\right)\cos \omega t\,dt
   = \int_0^\infty \langle a,e^{tA}Ra\rangle \cos \omega t\,dt  = \langle a,\widehat{S}_A(\omega)Ra\rangle.
   $$

\paragraph{$i=2$:} 
   $$
  \int_{-\infty}^\infty I_\delta^{(2)}(s,t)\,ds = \int_{-\infty}^\infty \left( \int_s^\infty  \langle a, e^{(u-s)A}B \gamma \rangle \psi(t-u)\,du\right)\psi(-s)\,ds.
   $$   
For $t > 0$ we have 
   $$
   \int_{-\infty}^\infty I_\delta^{(2)}(s,t)\,ds \to 
        \langle a, e^{tA}B \gamma \rangle 
                $$
as $\delta \to 0$.  
Also 
  $$
   \left| \int_{-\infty}^\infty I_\delta^{(2)}(s,t)\,ds\right|
      \le \sup_{|r| \le 2 \delta}\left|\langle a, e^{(t+r)A}B \gamma \rangle  \right| 
      \le \left(\sup_{|r| \le 2\delta}\|e^{rA^\ast}a\|\right) \|e^{tA}B \gamma\|.
     $$ 
Therefore by the dominated convergence theorem we have 
   $$
    \frac{1}{\pi} \int_0^\infty \left(\int_{-\infty}^\infty  I_\delta^{(2)} (s,t) \,ds\right)\cos \omega t\,dt
   \to  \frac{1}{\pi} \int_0^\infty \langle a, e^{tA}B \gamma \rangle  \cos \omega t\,dt = \langle a ,\widehat{S}_A(\omega)B\gamma \rangle
    $$
as $\delta \to 0$.

\paragraph{$i=3$: }
  $$
  \int_{-\infty}^\infty I_\delta^{(3)}(s,t)\,ds =  \int_{-\infty}^\infty \left(\int_s^\infty \langle a,e^{(u-s)A} B\gamma\rangle \,\psi_\delta(-u)du \right)  \psi_\delta(t-s)ds.
  $$
Note that $\left|\int_{-\infty}^\infty I_\delta^{(3)}(s,t)\,ds\right| \le \sup_{r \ge 0} |\langle a,e^{rA}B \gamma \rangle | < \infty$.  Also, if $0 < 2\delta < t$ then $\int_{-\infty}^\infty I_\delta^{(3)}(s,t)\,ds = 0$.       
Therefore by the dominated convergence theorem we have 
   $$
    \frac{1}{\pi} \int_0^\infty \left(\int_{-\infty}^\infty  I_\delta^{(3)} (s,t)\,ds\right)\cos \omega t\,dt
   \to  0
   $$
as $\delta \to 0$. 
   
\paragraph{$i=4$: } The substitution $u=s-t$ gives
   $$
  \int_{-\infty}^\infty I_\delta^{(4)}(s,t)\,ds  = \|\gamma\|^2 \int_{-\infty}^\infty \gamma_\delta(t-s)\psi_\delta(-s)ds = \int_{-\infty}^\infty I_\delta^{(4)}(s,-t)\,ds,
  $$   
and so
  \begin{align*}
  \frac{1}{\pi} \int_0^\infty \left(\int_{-\infty}^\infty  I_\delta^{(4)} (s,t)\,ds\right)\cos \omega t\,dt    
   & = \frac{1}{2\pi} \int_{-\infty}^\infty \left(\int_{-\infty}^\infty  I_\delta^{(4)} (s,t)\,ds\right)\cos \omega t\,dt \\
   & =  \frac{\|\gamma\|^2}{2\pi} \int_{-\infty}^\infty \int_{-\infty}^\infty \psi_\delta(t-s)\psi_\delta(-s) \cos \omega t \,dsdt
  \end{align*} Notice $\int_{-\infty}^\infty \psi_\delta(t-s)\psi_\delta(-s)ds = 0$ if $|t| > 2\delta$.  Also
  \begin{align*}
     \int_{-\infty}^\infty\int_{-\infty}^\infty \psi_\delta(t-s)\psi_\delta(-s)\,dsdt
     & = \int_{-\infty}^\infty \psi_{\delta}(-s)\left(\int_{-\infty}^\infty \psi_\delta(t-s)\,dt\right)ds\\
     & = \int_{-\infty}^\infty \psi_{\delta}(-s)\,ds = 1.
 \end{align*}   
Therefore
 \begin{align*}
    0 &  \le \frac{\|\gamma\|^2}{2\pi} - \frac{1}{\pi} \int_0^\infty \left(\int_{-\infty}^\infty  I_\delta^{(4)} (s,t)\,ds\right)\cos \omega t\,dt \\
    & =  \frac{\|\gamma\|^2}{2\pi} \int_{-\infty}^\infty \int_{-\infty}^\infty \psi_\delta(t-s)\psi_\delta(-s) \bigl(1-\cos \omega t \bigr) \,dsdt\\
    & \le  \frac{\|\gamma\|^2}{2\pi} \sup_{|t| \le 2\delta}\big|1-\cos \omega t\big|
  \end{align*} 
and so
  $$
    \frac{1}{\pi} \int_0^\infty \left(\int_{-\infty}^\infty  I_\delta^{(4)} (s,t)\,ds\right)\cos \omega t\,dt
   \to  \frac{\|\gamma\|^2}{2\pi}
   $$
as $\delta \to 0$.  

Together these four limits give \eqref{psd lim}, and the rest of Proposition \ref{prop psd xi} follows directly from the definition \eqref{lam2 omega}.  \qed

\section*{Acknowledgement}
The work of the second author was partially supported by National Sciences and Engineering Research Council of Canada Discovery grant 50503-10802.

\bibliographystyle{plain}
\bibliography{autoparametric}

    \end{document}